\def\year{2020}\relax
\documentclass[letterpaper]{article} 
\usepackage{aaai20}  
\usepackage{times}  
\usepackage{helvet} 
\usepackage{courier}  
\usepackage[hyphens]{url}  
\usepackage{graphicx} 
\usepackage{graphicx}  
\usepackage{url}            
\usepackage{booktabs}       
\usepackage{amsfonts}       
\usepackage{nicefrac}       
\usepackage{microtype}      

\usepackage{amsmath}

\usepackage{graphicx}
\usepackage{booktabs} 

\usepackage{siunitx}
\usepackage{graphicx}
\usepackage{url}
\usepackage{pdflscape}
\usepackage{multirow}
\usepackage{multicol}
\usepackage[english]{babel}
\usepackage{algorithmic}
\usepackage{algorithm2e}
\usepackage{comment}
\usepackage{latexsym, amsmath,amssymb, amsthm} 
\newtheorem*{theorem*}{Theorem}
\usepackage{graphics, epsfig, epstopdf, tabularx, footnote, booktabs,pbox, tabularx}
\usepackage{amsfonts}

\usepackage{enumitem}

\usepackage{amsmath}
\usepackage{pdflscape}
\usepackage{verbatim}
\usepackage{amsfonts}

\SetKwFor{ParallelForEach}{for each}{in parallel do}{endfor}

\newcommand{\Reals}{\mathbb{R}}

\newcommand{\ip}[2]{\left \langle #1 , \;  #2 \right \rangle}

\newcommand{\R}{\mathbf{R}}

\newcommand{\E}{\mathbf{E}}

\newcommand{\norm}[1]{\left\lVert#1\right\rVert}

\usepackage{epstopdf}
 
\usepackage{amsthm} 
\theoremstyle{plain}
\newtheorem{theorem}{Theorem}

\newtheorem{lemma}[theorem]{Lemma}

\newtheorem{assumption}[theorem]{Assumption}

\theoremstyle{definition}
\newtheorem{definition}[theorem]{Definition}

\theoremstyle{remark}

\usepackage{dcolumn}

\newcolumntype{d}[1]{D{.}{\cdot}{#1} }

\newcommand{\inTR}[1]{}

\SetCommentSty{\small\sf}

\usepackage{algorithmic}

\DeclareGraphicsExtensions{.pdf,.jpeg,.png}
\usepackage{algorithmic}
\usepackage{array}
\usepackage{url}

\newcommand{\cX}{{\cal X}}
\newcommand{\cY}{{\cal Y}}

\newcommand{\vzero}{{\bf 0}}

\renewcommand{\c}{{\bf c}}

\newcommand{\e}{{\bf e}}

\newcommand{\s}{{\bf s}}
\newcommand{\x}{{\bf x}}

\renewcommand{\v}{{\bf v}}

\newcommand{\y}{{\bf y}}
\newcommand{\A}{{\bf A}}

\newcommand{\C}{{\bf C}}

\renewcommand{\E}{{\bf E}}

\renewcommand{\L}{{\bf C}}
\newcommand{\M}{{\bf M}}
\renewcommand{\P}{{\bf P}}

\newcommand{\mopt}{\text{OPT}_{\!f}}

\newcommand{\hatSr}{\hat{S}_{\rm row}}
\newcommand{\hatSc}{\hat{S}_{\rm column}}

\DeclareMathOperator*{\minimise}{\mathrm{minimise}}

\title{Pursuit of Low-Rank Models of Time-Varying Matrices \\ Robust to Sparse and Measurement Noise}

\author{
  Albert Akhriev, Jakub Marecek, Andrea Simonetto \\
  IBM Research -- Ireland \\
 \{albert\_akhriev, jakub.marecek\}@ie.ibm.com, andrea.simonetto@ibm.com 
}

\begin{document}

\maketitle

\begin{abstract}
In tracking of time-varying low-rank models  of time-varying matrices,
we present a method robust to both uniformly-distributed measurement noise
and arbitrarily-distributed ``sparse'' noise.
In theory, we bound the tracking error.
In practice, our use of randomised coordinate descent is scalable and
allows for encouraging results on changedetection.net, a benchmark.
\end{abstract}

\section{Introduction}

Dimension reduction is a staple of Statistics and Machine Learning. 
In principal component analysis, its undergraduate-textbook version, possibly correlated observations  
are transformed to a combination of linearly uncorrelated variables, called principal components.
Often, a low number of principal components suffice
for the so-called low-rank model to represent the phenomenon observed.
Notoriously, however, a small amount of noise can change the 
principal components considerably.
A considerable effort has focussed on the development of robust approaches to principal component analysis (RPCA).
Two challenges remained: robustness to both sparse and non-sparse noise and theoretical guarantees in the time-varying setting.

We present the pursuit of time-varying low-rank models of time-varying matrices, which is
robust to both dense 
uniformly-distributed measurement noise and sparse arbitrarily-distributed noise.
Consider, for example, background subtraction problem in Computer Vision, where one wishes to distinguish fast-moving foreground objects from slowly-varying background in video data \cite{6180173}.
There, a matrix represents a constant number of frames of the video data, flattened to one row-vector per frame.
At any point in time, the low-rank model is captured by a short-and-wide matrix.
The time-varying low-rank model
makes it possible to capture slower changes, e.g., lighting conditions slowly changing with the cloud cover.
There may also be slight but rapid changes,
e.g., leaves of grass moving in the wind,
which could be captured by the uniformly-distributed dense noise.
Finally, the moving objects are captured by the sparse noise.
Clearly, low-rank modelling has wide-ranging applications beyond Computer Vision, 
wherever one needs to analyse high-dimensional streamed data
and flag abnormal observations to operators, 
while adapting the model of what is normal over time. 

\begin{table*}[t]
\caption{A comparison of our approach against five of the best-known RPCA implementations and the recent OMoGMF, featuring the F1 score on the baseline category of http://changedetection.net and mean run time (in seconds per input frame, single-threaded) on the ``baseline/highway'' video-sequence of the same benchmark.}
\label{relatedwork}
\centering 
\begin{tabular}{lllrr}
\toprule
Method & Model & Guarantees & {F1} & {Run-time} \\
\midrule
LRR\_FastLADMAP  & Low-rank + Sparse & Off-line: limit point KKT  & 0.36194 & 4.611\\
MC\_GROUSE & Low-rank + Sparse, $L_2$ &  --- & 0.31495 & 10.621 \\
OMoGMF   & GMM(Low-rank) + Sparse &  --- & 0.72611 & 0.123 \\
RPCA\_FPCP & Low-rank + Sparse &  --- & 0.37900 & 0.504  \\
ST\_GRASTA & Rank-1 + Sparse, $L_1$ & ---& 0.42367 & 3.266 \\
TTD\_3WD & Low-rank + Turbulence + Sparse & Off-line: limit point feasible & 0.40297 & 10.343 \\
Our approach & Low-rank + Uniform + Sparse &  \textbf{On-line: tracking error} & \textbf{0.80254} &  \textbf{0.103} \\
\bottomrule
\end{tabular}
\end{table*}

Our contributions are as follows:
\begin{itemize}
    \item we extend the low-rank + sparse model to low-rank + dense uniformly-distributed noise + sparse, where low-rank can be time-varying 
    \item we provide an algorithm with convergence-rate guarantees for the time-invariant case 
    \item we provide an algorithm with guarantees for the time-varying case. In Theorem \ref{thm:gam-msc-mss-proof}, 
we bound the tracking error of an algorithm for any low-rank factorisation problem for the first time. That is: we show that a sequence of approximately optimal costs eventually reaches the optimal cost trajectory.
\item we improve upon the statistical performance  of RPCA approaches on changedetection.net of \cite{goyette2012changedetection}, a well-known benchmark: the F1 score across 6 categories of changedetection.net improves by 28\%, from 0.44643 to 0.57099. On the baseline category, it is 0.80254.
\item we improve upon run time per frame of the same RPCA approaches, as detailed in Table~\ref{relatedwork}. 
Compared to TTD\_3WD, to give an example of a method which is still considered efficient in the literature, our single-threaded implementation is 103 times faster.
\end{itemize}


\section{Related Work}

Traditional approaches to robustness in  low-rank models \cite[to name some of the pioneering work]{candes2009exact}  
are based on a long history of work in robust statistics \cite{huber1981robust}. In such approaches \cite{candes2011robust,NIPS2013_5135,guo2014online,Mardani2013}, sometimes known as ``Low-rank + Sparse'', one balances the number of samples of the ``sparse'' noise and the rank of the model, or the 
nuclear norm as a proxy for the rank. 
There are a number of excellent  implementations, including some focused on the incremental update
\cite[e.g.]{Lin11,He11,Balzano13,Oreifej13,Meng13,Rodriguez13,dutta2017problem,dutta2017batch,8412568,8425657,8425658,8417980,Yong18}. 
In our comparison, we focus five of the best-known implementations and one very recent one.
LRR\_FastLADMAP \cite{Lin11}, RPCA\_FPCP \cite{Rodriguez13}, and MC\_GROUSE \cite{Balzano13} use the low-rank + sparse model.   
ST\_GRASTA \cite{He11} uses rank-1 + sparse. 
TTD\_3WD \cite{Oreifej13} uses low-rank + turbulence + sparse. 
The most recent formulation we consider is 
OMoGMF \cite{Yong18}, which utilises a Gaussian mixture model (GMM) structure over the low-rank model, plus sparse noise on top.
We refer to \cite{bhojanapalli2016global,NIPS2016_6517,jain2017non,boumal2018global,bhojanapalli2018smoothed} for the present-best theoretical analyses in the off-line, time-invariant case, but stress that no guarantees have been known for the on-line, time-varying case.
We refer to the recent handbook \cite{bouwmans2016handbook}  
and to the August 2018 special issue of the Proceedings of the IEEE \cite{8425660} for up-to-date surveys.

\section{Problem Formulation}
\label{sec:problem_statement}

Consider $N$ streams with $n$-dimensional measurements, coming from $N$ sensors with uniform sampling period $h$ from $t_k$ till $t_k + h T$ (possibly with many missing values), packaged in a (possibly partial) matrix $\M_k \in \Reals^{T \times n N}$. Every time a new observation comes in, its flattening is added at the bottom row to the matrix and the first row is discarded. In this way, the observation matrix slowly varies over time, i.e., $\M_{k+1}$ is different from $\M_k$, in general.

It is natural to assume that any row $d$ may resemble a linear combination of $r \ll T$ prototypical rows.
Prior to the corruption by sparse noise, we assume that there exists $\R_k \in \Reals^{r \times n N}$, such that  flattened observations $\x_d \in \Reals^{1 \times nN}$ are
\begin{align}
\label{assumption1}
\x_d = \c_d \R_k + \e_d,
\end{align}
where the row vector $\c_d \in \Reals^{1 \times r}$ weighs the rows of matrix $\R_k$, while $\e_d\in \Reals^{1 \times nN}$ is the noise row vector, where each entry be uniformly distributed between known, fixed $-\Delta$ and $\Delta$. 
Further, this  formulation~\eqref{assumption1} is extended 
towards the contamination model \cite{huber1981robust},
where ``sparse errors'' replace readings of some of the sensors. 
That is: Either we receive a measurement belonging to our model, or not:
\begin{equation}
(\x_d)_{i} = ({\bf{1}}_n-\mathbb{I}_{i,k}) \circ \left[(\c_d \R_k)_i + (\e_d)_i\right] + \mathbb{I}_{i,k}\circ\s_i,
\end{equation}
where index $i$ enumerates sensors, $\s_i \in \Reals^{1 \times n}$ is a generic noise vector, while the Boolean vector $\mathbb{I}_{i,k}\in\{0,1\}^n$ has entries that are all zeros or ones depending on whether we receive a measurement belonging to our model or not. The operation $\circ $ represents element-wise multiplication. 

Considering the matrix representation, we assume that the matrix $\M_k$ can be decomposed into slowly varying low-rank model ($\L_k \R_k$) and additive deviation ($\E_k$) from the model comprising noise and anomalies:
\begin{equation}
\M_k = \left[\begin{array}{c} 
\dots \\ \hline \x_d \\ \hline \dots
\end{array}\right] = \L_k \R_k + \E_k,
\end{equation}
where $T$ is the number of samples stacked in rows of matrix $\M_k$, $r$ is the number of prototypes in the low-rank approximation, $\x_d$ is a $d$-th row-vector in matrix $\M_k$, $\L_k \in \Reals^{T\times r}$ and $\E_k \in \Reals^{T \times n N }$ are the matrices incorporating the coefficient vectors $\c_d$'s and noise $\e_d$'s as $\L_k = [\dots; \c_d\,; \dots]$, and $\E_k = [\dots; \e_{d}\,; \dots]$, respectively.

The missing entries in $\M_k$ can represent either really absent data or outliers, such as moving objects in the case of video-processing applications. One can assume that normal behaviour exhibits certain regularity, which could be captured by a low-rank structure, and that events or anomalies are sparse across both time and space. The sparsity should be construed quite loosely, for example, comprising dense blobs of pixels moving coherently in video data, while occupying a relatively small fraction of image pixels in total. This notion of anomaly detection is widely used in monitoring streamed data, event recognition, and computer vision. 

If we can identify the low-rank model, any deviation from the  measurement model~\eqref{assumption1}  can be interpreted as an anomaly or event. 
When there are few measurements for which $\mathbb{I}_{i,k} = {\bf{1}}_n$ and those are different from standard measurements, i.e., 
the aggregated $\mathbb{I}_k \in\{0,1\}^{nN}$, which stacks all the individual $\mathbb{I}_k$ for a specific time $k$, is sparse,
and samples of $s_i$ fall outside of some range $[\underline{M}_{k,ij}, \overline{M}_{k,ij}]$ (defined below), it is possible to identify samples of $s_i$ perfectly.
In this paper, we provide a way to detect such anomalies, i.e., measurements for which $\mathbb{I}_{i,k} = {\bf{1}}_n$. 
Hence, we are effectively proposing a principal component pursuit algorithm robust to uniform and sparse noise.


We compute matrices $\L_k$ and $\R_k$ by resorting to a low-rank approximation of the matrix $\M_k$ with an explicit consideration of the uniformly-distributed 
error in the measurements.  Let $M_{k,ij}$ be the $(i,j)$ element of $\M_k$. Consider the interval uncertainty set $[M_{k,ij} -\Delta, M_{k,ij} + \Delta]$ around each observation. Finding $(\L_k,\R_k)$ can be seen 
as matrix completion with 
element-wise lower bounds $\underline{M}_{k,ij} := M_{k,ij} -\Delta$ and
element-wise upper bounds $\overline{M}_{k,ij} := M_{k,ij} + \Delta$. 
Let $\L_{k,i:}$ and $\R_{k,:j}$ be the $i$-th row and $j$-th column of $\L_k$ and $\R_k$, respectively. With Frobenius-norm regularisation, the completion problem we solve is:
\begin{equation}
\minimise_{\L_k\in \Reals^{T\times r}, \; \R_k\in \Reals^{r\times n N}} f(\L_k,\R_k; \M_k), 
\label{eq:Specific}
\end{equation}
where:
\begin{align}\label{defSpecific} 
f(\L_k,\R_k; \M_k) & {\,\,:=} &
\textstyle{\tfrac{1}{2}\sum_{(ij)} \ell(\underline{M}_{k,ij}-\L_{k,i:}\R_{k,:j})}
\nonumber \\
{} & {+} &
\textstyle{\tfrac{1}{2}\sum_{(ij)} \ell(\L_{k,i:}\R_{k,:j}-\overline{M}_{k,ij})}
\nonumber \\ 
{} & {+} & \tfrac{\nu}{2}\|\L_k\|_{F}^2 + \tfrac{\nu}{2}\|\R_k\|_{F}^2, \quad \quad \quad \; \;
\end{align}
where $\ell: \Reals \to \Reals$ is the square of the maximum of the two-element set composed of the argument and 0, as detailed in Section~``A Derivation of the Step Size''
of the appendix, and $\nu >0$ is a weight. 

Our only further assumption is that we have the element-wise constraints on all elements of the matricial variable: 
\begin{assumption}
    \label{ass:partition}
    For each $(i,j)$ of $\M_k$ there is a finite element-wise upper bound $\overline{M}_{k,ij}$ and a finite
    element-wise lower bound $\underline{M}_{k,ij}$.
\end{assumption}
This assumption is satisfied even for any missing values at $ij$ when the measurements lie naturally in a bounded set, e.g.,  $[0, 255]$ in many computer-vision applications.     

\section{Proposed Algorithms}
\label{sec:method}

In this section, we first present the overall schema of our approach in Algorithm \ref{alg:complete}. Second, we present  Algorithm~\ref{alg:SCDM} for on-line   inequality-constrained matrix completion, a crucial sub-problem.

\subsection{The Overall Schema}

Overall, we interleave 
the updates to the low-rank model via the inequality-constrained matrix completion,
detection of sparse noise, 
and updating of the inputs to the inequality-constrained matrix completion,
which disregards the sparse noise.

At each time step, we acquire new measurements $\x_d$ 
and compute their projection coefficients onto the low-rank subspace as
\begin{equation}
\v = \arg\min_{\v \in \mathbb{R}^{1 \times r}} \|\x_d - \v \R_{k-1}\|_p,
\label{proj}
\end{equation}
where $p$ can be the $1,2,\infty$ norm, or the $0$ pseudo-norm. 
Since for a very large number of sensors, even solving~\eqref{proj} can be challenging, we subsample $\x_d$ by picking only a few sensors uniformly at random. Let $i\in \tilde{\mathcal{N}}$ be the sampled sensors, with $|\tilde{\mathcal{N}}|= \tilde{N}$. We form a low-dimensional measurement vector $\tilde{\x}_d \in \mathbb{R}^{1\times n\tilde{N}}$ and solve the subsampled:
\begin{equation}\label{sampledproj}
\tilde{\v} = \arg\min_{\v \in \mathbb{R}^{1 \times r}} \|\tilde{\x}_d - \v (\R_{k-1}^i)_{i \in \tilde{\mathcal{N}}}\|_p,
\end{equation}
where $(\R_{k-1}^i)_{i \in \tilde{\mathcal{N}}} \in \mathbb{R}^{r\times n \tilde{N}}$ is the matrix whose columns corresponds to the   sensors, which are sampled uniformly at random. 
Solving~\eqref{sampledproj} yields solutions $\tilde{\v}$ such that the norm $\|\v - \tilde{\v}\|_p$ is very small, while being considerably less demanding computationally.  

Once the projection coefficients $\v$ have been computed, we can 
compute the discrepancy between the measurement $(\x_d)_i$ coming from sensor $i$  
and our projection \eqref{sampledproj}, 
$\|(\x_d)_i - (\v \R_{k-1})_i\|_p$, also known as the residual for sensor $i$.
We use the residuals in a two-step thresholding procedure inspired by \cite{Malistov2014}. 
In the first step, we use residuals 
to compute a coefficient $\lambda > 0$. 
In the second step, we consider the individual residuals as samples of an empirical distribution,
and take the value at risk (VaR) at $\lambda$ as a threshold.
We provide details in \cite{akhriev2019deep,akhriev2018pursuit}.
The test as to whether residual at each sensor is below the threshold results in a binary map, suggesting whether the observation of each sensor is likely to have come from our model or not.
For a positive value at $i$ in the map, the measurement $(\x_d)_i$ is kept in $\M_k$. Otherwise, it is discarded. 

\begin{algorithm}[t]
    \caption{Pursuit of low-rank models of time-varying matrices robust to both sparse and measurement noise.}
    \label{alg:complete}
    \begin{algorithmic}[1]
        \item[] \textbf{Input}: Initial matrices $(\L_{0}, \R_{0})$, rank $r$ 
        \item[] \textbf{Output}: $(\L_{k}, \R_{k})$ and events for each $k$ 
        \FOR{each time ${t_k}: k=1,2,\dots,\, t_{k+1}-t_k = h$}
        \STATE{acquire new measurements $\x_d$}
        \STATE{subsample $\x_d$ uniformly at random to obtain $\tilde{\x}_d$ }
        \STATE{compute $\tilde{\v}$ via the subsampled projection ~\eqref{sampledproj}}
        \FOR{each sensor $i$ {\bf in parallel}}
        \STATE{compute residuals $r_i = \|(\x_d)_i - (\tilde{\v} \R_{k-1})_i\|$ }
        \ENDFOR
        \STATE{compute $\lambda$ as a function of $\{r_i\}_i$ as described in the appendix}
        \STATE{compute $T$ as a value at risk at $\lambda$ of $\{r_i\}$} 
        \STATE{initialise $\y$ as a boolean all-False vector of same dimension as $\x_d$ }
        \FOR{each sensor $i$ {\bf in parallel}}
        \IF{$r_i < T$}
        \STATE{set $\y_i$ to True, as value at sensor $i$ is likely to come from our model}
        \STATE{add $(\x_d)_i$ to $\M_k$}
        \ENDIF
        \ENDFOR
        \STATE{compute $(\L_{k}, \R_{k})$ via Algorithm~\ref{alg:SCDM} with rank $r$}
        \ENDFOR
        \STATE{\bf{return} $(\L_{k}, \R_{k}, \y)$ }
    \end{algorithmic}
\end{algorithm}

\subsection{On-line Matrix Completion}

Given $\M_k$, we utilise \emph{inequality-constrained matrix completion}, to estimate the low-rank approximation 
$(\L_{k}, \R_{k})$    of the input matrix considering interval uncertainty sets. 
Clearly, solving the non-convex problem \eqref{eq:Specific} for non-trivial dimensions of matrix $\M_k$ 
to a non-trivial accuracy at high-frequency 
requires careful algorithm design. 
We propose an algorithm that tracks the low-rank $\R_k$ over time, increasing the accuracy of the solution of~\eqref{eq:Specific} while new observations are brought in, and old ones are discarded. 
In particular, we propose the on-line alternating parallel randomised block-coordinate descent method 
summarized in Algorithm~\ref{alg:SCDM}.

For each input $k$, the previously-found approximate solutions $(\L_{k-1}, \R_{k-1})$,
are updated based on the new observation matrix $\M_k$,
the correspondingly-derived element-wise lower and upper bounds $\underline{M}_{k,ij}, \overline{M}_{k,ij}$, and
the desired rank $r$.
The update is computed using the alternatig least squares (ALS) method,
which is based on the observation that while $f$ \eqref{eq:Specific} is not convex jointly
in $(\L_k,\R_k)$, it is convex in $\L_k$ for fixed $\R_k$ and in $\R_k$ for fixed $\L_k$.
The update takes the form of a sequence $\{(\L_{k}^{T,\tau}, \R_{k}^{T,\tau})\}$ of solutions, which are progressively more accurate.
If we could run a large number of iterations of the ALS, we would be in an off-line mode.
In the on-line mode, we keep the number of iterations small, and apply the final update 
based on $\M_k$ at time
$t_{k + 1}$, when the next observation arrives.     

The optimisation in each of the two alternating least-squares problems is based on parallel block-coordinate descent, as reinterpreted by \cite{Nesterov2012}.
Notice that in Nesterov's optimal variant, one requires the
the modulus of Lipschitz continuity restricted to the sampled coordinates
\cite[Equation 2.4]{Nesterov2012} to compute the step $\delta$.
Considering that the modulus is not known \emph{a priori}, we maintain an estimate $W_{i\hat{r}}^{T,\tau}$ of the modulus of Lipschitz continuity  restricted to the $\L_{k,i \hat r}^{T,\tau}$
sampled, 
and estimate $V_{\hat{r} j}^{T,\tau}$ of the modulus of Lipschitz continuity  restricted to the $\R_{k, \hat r j}^{T,\tau}$ sampled.
We refer to the appendix for the details of the estimate and to  \cite{Nesterov2012} for a high-level overview.

Overall, when looking at Algorithm~\ref{alg:SCDM}, notice that there are several nested loops. The counter for the update of the input is $k$.
For each input, we  consider factors $\L$ and $\R$ as the optimisation variable alternatingly, with counter $T$.
For each factor, we take a number of block-coordinate descent steps, with the blocks sampled randomly; the counter for the block-coordinate steps is $\tau$.
In particular, in Steps 3--8 of the algorithm, we fix $\R_k^{T,\tau}$, choose a  random $\hat{r}$ and a random set $\hatSr$ of rows of $\L_k$, and, in parallel for $i \in \hatSr$, update  $\L_{k,i\hat{r}}^{T,\tau+1}$ to $\L_{k,i\hat{r}}^{T,\tau} + \delta_{i\hat{r}}$,
where the step is:
\begin{equation} \label{eq:delta_L}
\delta_{i\hat{r}}:= - \langle \nabla_{\L_k} f(\L_k^{T,\tau},\R_k^{T,\tau}; \M_k), \P_{i\hat{r}}\rangle / W_{i\hat{r}}^{T,\tau},
\end{equation}
and $\P_{i\hat{r}}$ is the $n \times r$ matrix with $1$ in entry $(i\hat{r})$ and zeros elsewhere.
The computation of $\langle \nabla_{\L_k} f(\L_k^{T,\tau},\R_k^{T,\tau}; \M_k), \P_{\hat{r}j}\rangle$ can be simplified considerably, 
as explained in 
in Section ``A Derivation of the Step Size'' 
of the appendix. 

Likewise, in Steps 9--14, we fix  $\L_k^{T,\tau+1}$, choose a  $\hat{r}$ and a random set $\hatSc$ of columns of $\R_k$, and, in parallel
for $j \in \hatSc$, update $R_{k,\hat{r}j}^{T,\tau+1}$ to $R_{k,\hat{r}j}^{T,\tau} + \delta_{\hat{r}j}$, where the step is:
\begin{equation} \label{eq:delta_R}
\delta_{\hat{r}j}:= - \langle \nabla_{\R_k} f(\L_k^{T,\tau+1},\R_k; \M_k), \P_{\hat{r}j}\rangle / V_{\hat{r}j}^{T,\tau},
\end{equation}
and $\P_{\hat{r}j}$ is the $r\times m$ matrix with $1$ in entry $(\hat{r}j)$ and zeros elsewhere.
Again, the computation of $\langle \nabla_{\R_k} f(\L_k^{T,\tau+1},\R_k; \M_k), \P_{\hat{r}j}\rangle$ can be simplified.

\begin{algorithm}[t]
    \caption{On-line inequality-constrained matrix-completion via randomised coordinate descent.}
    \label{alg:SCDM}
    \begin{algorithmic}[1]
        \item[] \textbf{Input}: updated $\M_k$, $\underline{M}_{k,ij}, \overline{M}_{k,ij}$,
        previous iterate $(\L_{k-1}, \R_{k-1})$, 
        rank $r$, limit $\overline{\tau}$
        \item[] \textbf{Output}: $(\L_{k}, \R_{k})$ 
        \STATE \label{alg:stp:initialpoint} Initialise: $(\L_{k}^{0,0} = \L_{k-1}, \R_k^{0,0} = \R_{k-1})$, $T = 0$
        \WHILE{$\M_{k+1}$ is not available} 	
        \FOR{$\tau=0,1,2,\dots,\overline{\tau}$} 
        \STATE choose $\hatSr \subseteq \{1,\dots,m\}$
        \FOR{$i\in \hatSr$ {\bf in parallel}}
        \STATE choose $\hat r \in \{1,\dots,r\}$ uniformly at random
        \STATE compute $\delta_{i \hat r}$ using formula \eqref{eq:delta_L}
        \STATE update $\C_{k, i \hat r}^{T,\tau+1}  \leftarrow  \C_{k, i \hat r}^{T,\tau} + \delta_{i \hat r}$
        \ENDFOR
        \ENDFOR 			
        \FOR{$\tau=0,1,2,\dots,\overline{\tau}$}
        \STATE choose $\hatSc \subseteq \{1,\dots,n\}$ uniformly at random
        \FOR{$j\in \hatSc$ {\bf in parallel}}
        \STATE choose $\hat r \in \{1,\dots,r\}$ uniformly at random
        \STATE compute $\delta_{\hat{r}j}$ using \eqref{eq:delta_R}
        \STATE update $\R_{k, \hat{r}j}^{T,\tau+1} \leftarrow \R_{k, \hat{r}j}^{T,\tau} + \delta_{\hat{r} j}$
        \ENDFOR
        \ENDFOR
        \STATE set: $\L_{k}^{T+1,0} = \L_{k}^{T,\overline{\tau}+1}$, $\R_{k}^{T+1,0} = \R_{k}^{T,\overline{\tau}+1}$
        \STATE update: $T = T+1$
        \ENDWHILE
        \STATE \textbf{return} $\C_k = \C_k^{T,0}$, $\R_k = \R_k^{T,0}$ 
    \end{algorithmic}
\end{algorithm}

\section{Convergence Analysis}
\label{sec:anal}

For the off-line inequality-constrained 
matrix completion problem \eqref{eq:Specific}, \cite{marecek2017matrix} proposed an algorithm similar to Algorithm \ref{alg:SCDM} and presented a convergence result, which states that the method is monotonic and, with probability 1, converges to the so-called bistable point, i.e.,  $\liminf_{T \to \infty} \|\nabla_{\L} f(\L^{\tau},\R^{\tau}; \M)\| = 0,$ and  $\liminf_{T \to \infty} \|\nabla_{\R} f(\L^{\tau},\R^{\tau}; \M)\| = 0$. 
Here, we need to show the rate of convergence to the bistable point and a distance of the bi-stable point to an optimum $f^*$:


\begin{theorem}[]
    \label{thm:gam-msc-mss-proof}
    There exists $\overline \tau > 0$, such that
    Algorithm~\ref{alg:SCDM} with the initialization to all-zero vector after 
    at most $T = O(\log\frac{1}{\epsilon})$ steps has $f( \L^T, \R^T) \leq f^\ast + \epsilon$
    with probability 1.
\end{theorem}

The proof is available on-line the appendix and should not be surprising, in light of \cite{bhojanapalli2016global,NIPS2016_6517,jain2017non,boumal2018global,bhojanapalli2018smoothed}.

Building upon this, we can prove a bound on the error in the on-line regime.
In particular, we will show that Algorithm~\ref{alg:SCDM} generates a sequence of matrices $\{(\L_k, \R_k)\}$ that in the large limit of $k \to \infty$ guarantees a bounded tracking error, i.e., $f(\L_k, \R_k; \M_k) \leq f(\L_k^*, \R_k^*; \M_k) + E$. The size of the tracking error $E$ depends on how fast the time-varying matrices change:

\begin{assumption}\label{as:varying}
    The variation of the observation matrix $\M_k$ at two subsequent instant $k$ and $k-1$ is so to guarantee that
    $$
    |f(\L_k,\R_k; \M_k) - f(\L_k,\R_k; \M_{k-1})| \leq e,
    $$
    for all instants $k>0$. 
\end{assumption}

Now, let us bound the error in tracking, i.e., 
when $M_k$ changes over time and we run only a limited number of iterations $\tau$ of our algorithm per time step, before obtaining new inputs. 

\begin{theorem}
    \label{theorem.inexact}
    Let Assumptions~\ref{ass:partition} and \ref{as:varying} hold. 
    Then with probability 1, Algorithm~\ref{alg:SCDM} starting from an all-zero matrices generates a sequence of 
    matrices $\{(\L_k, \R_k)\}$ for which
    \begin{multline}
    f(\L_k, \R_k; \M_k) - f(\L_k^*, \R_k^*; \M_k) \leq \\ \eta_0 (f(\L_{k-1}, \R_{k-1}; \M_{k-1}) - f(\L_{k-1}^*, \R_{k-1}^*; \M_{k-1})) + \eta_0 e,
    \notag
    \end{multline}
    where $\eta_0<1$ is a constant.
    In the limit, 
    \begin{equation}\label{eq.asympt_error}
    \limsup_{k\to \infty} f(\L_k, \R_k; \M_k) - f(\L_k^*, \R_k^*; \M_k) \leq \frac{\eta_0 e}{1-\eta_0} =: E.
    \notag
    \end{equation}
\end{theorem}


In other words, as time passes, our on-line algorithm generates a sequence of approximately optimal costs 
that eventually reaches the optimal cost \emph{trajectory}, up to an asymptotic bound. 
We bound from above the maximum discrepancy between the approximate optimum and 
the true one at instant $k$, as $k$ goes to infinity. 
The convergence to the bound is linear and the rate is $\eta_0$, and depends on the properties of the cost function, 
while the asymptotic bound depends on how fast the problem is changing over time. 

This is a \emph{tracking} result: we are pursuing a time-varying optimum by a finite number 
of iterations $\tau$ per time-step. 
If we could run a large number of iterations per each time step, then we would be back to a off-line case and we would not have a tracking error. 
This may not, however, be possible in settings, where inputs change faster than one can compute an
iteration of the algorithm.

\section{Experimental Evaluation}
\label{sec:exp}


We have implemented Algorithms \ref{alg:complete} and \ref{alg:SCDM} in C++, and released the implementation\footnote{ https://github.com/jmarecek/OnlineLowRank} under Apache License 2.0.  
Based on limited experimentation, we have decided on the use of a time window of $T = 35$, rank $r = 4$, 
and half-width of the uniform noise $\Delta = 5$.
We have 
used dual simplex from IBM ILOG CPLEX 12.8 as a linear-programming solver for solving solving \eqref{sampledproj} in Algorithm \ref{alg:complete}.
To initialise the $\L_0$ and $\R_0$ in Algorithm \ref{alg:complete}, we have used the matrix completion 
of Algorithms \ref{alg:SCDM} with 1 epoch per frame for 3 passes on each video (4,000 to 32,000 frames), starting from all-zero matrices.
We note that in real-world deployments, such an initialisation may be unnecessary, as the
the number of frames processed will render the initial error irrelevant.

\begin{figure}[t!]
\centering
\includegraphics[width=.95\columnwidth]{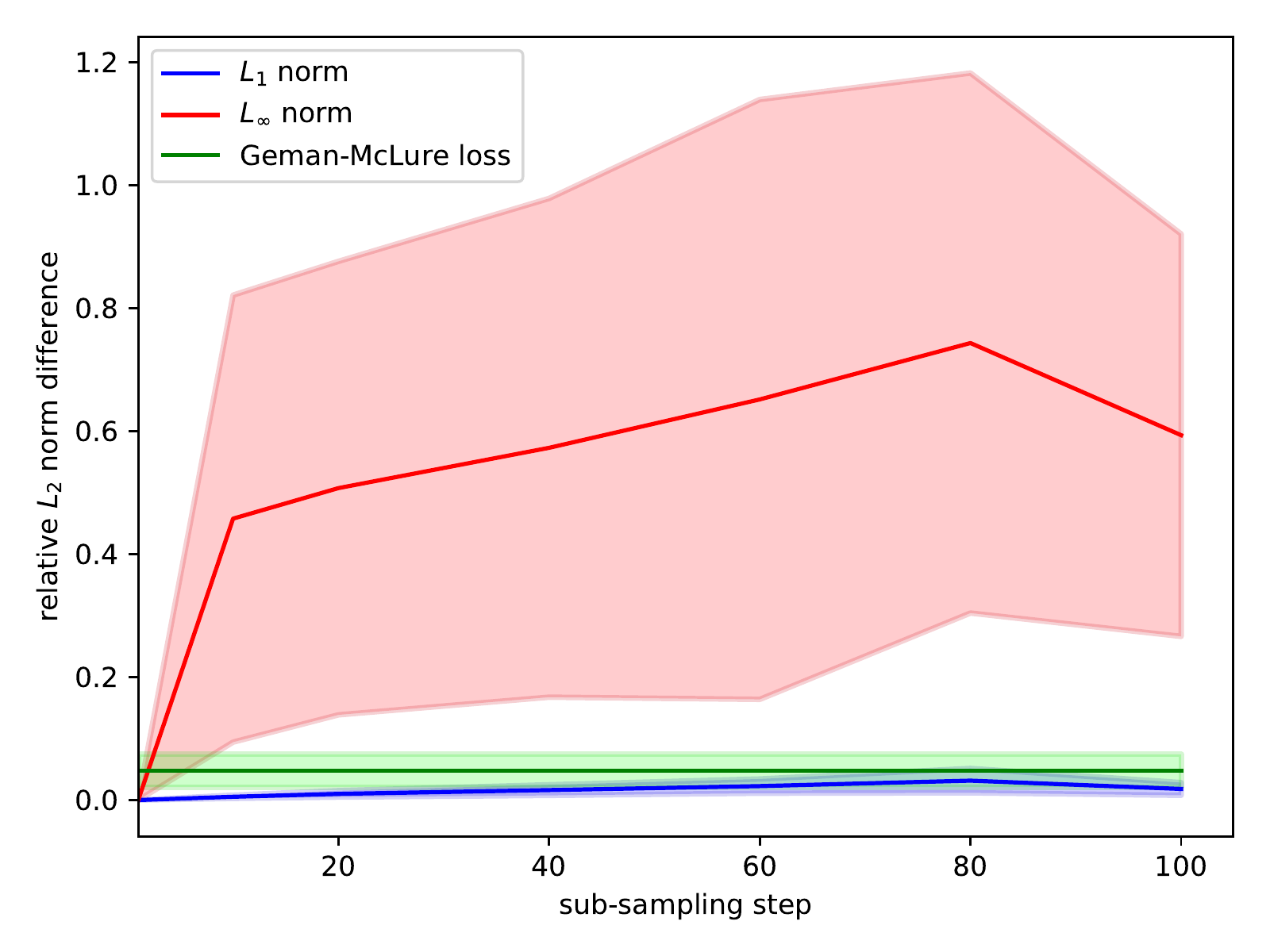}
\includegraphics[width=.95\columnwidth]{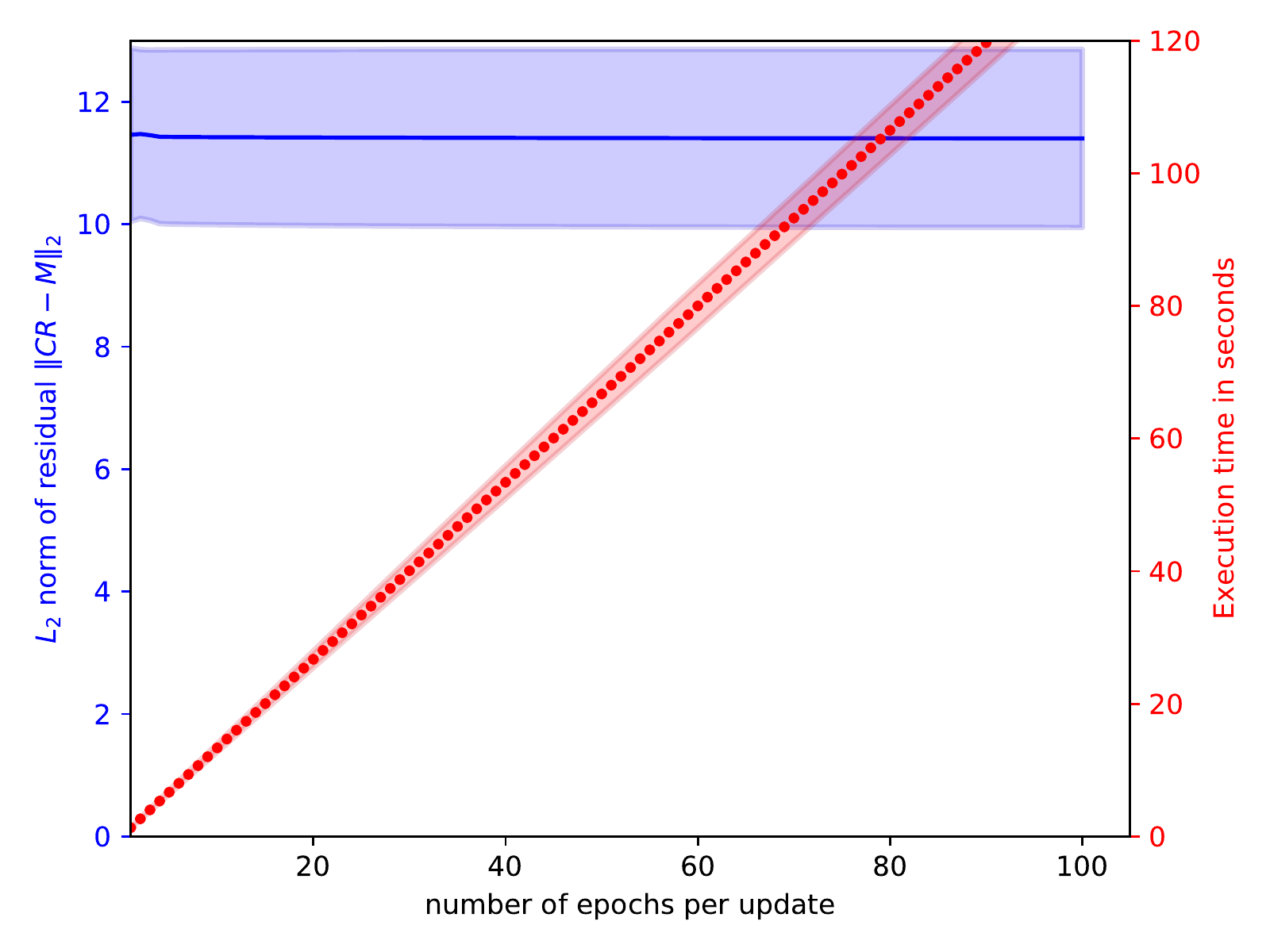}
\caption{Top: Effects of subsampling in the projection \eqref{sampledproj}.
    Bottom: Performance of Algorithm \ref{alg:SCDM} as a function of the number of epochs per update.}
\label{fig1}
\end{figure}

First, let us highlight two aspects of the performance of the algorithm. 
In particular, on the top in Figure~\ref{fig1}, we illustrate the effects of the subsampling on the projection \eqref{sampledproj}.
For projection in $L_1$ and $L_\infty$, we present the $L_2$ norm of the difference $\tilde{\v} - \v$
as a function of the sample period of the subsampling  \eqref{sampledproj}, 
where
$\v$ is the true value obtained in \eqref{proj} without subsampling and
$\tilde{\v}$ is the value obtained in \eqref{sampledproj} with subsampling,
and the sample period is the ratio of the dimensions of $\x_d$ and $\tilde{\x}_d$.
It is clear that $L_1$ is very robust to the subsampling.
This corroborates the PAC bounds of \cite{marecek2018low}
and motivated our choice of $L_1$ with a sampling period of 100 pixels in the code.
For completeness, we also present the performance of the Geman-McLure loss \cite{Sawhney96},
where we do not consider subsampling,
relative to the performance of $L_1$ norm without subsampling.

Next, on the bottom in Figure~\ref{fig1},
we showcase the $L_2$ norm of residual $\L_k\R_k - \M_k$ and the per-iteration run-time of a single-threaded
implementation as a function of the number of epochs per update.
Clearly, the decrease in the residual is very slow beyond one epoch per update,
due to the reasonable initialisation.
On the other hand, there is a linear increase in per-iteration run-time with the
number of epochs of coordinate descent per update. 
This motivated our choice of 1 epoch per update, which allows for real-time
processing at 10 frames per second \textit{without} parallelisation, which can further improve performance as suggested in Algorithm \ref{alg:SCDM}.

\begin{table*}[tp!]
\caption{Results of our Algorithm~\ref{alg:SCDM}, compared to 6 other approaces on the ``baseline'' category of http://changedetection.net, evaluated on the 6 performance metrics of (Goyette et al. 2012). For each performance metric, the best result across the presented methods is highlighted in bold. }
\label{tab:all-baseline}
\begin{center}
{\footnotesize
\begin{tabular}{lcccccc}
\toprule
Approach / Performance metric & Recall & Specificity & FPR
& FNR & Precision & F1     \\
\midrule
LRR\_FastLADMAP \cite{Lin11}  & 0.74694 & 0.93980 & 0.06021 & 0.25306 & 0.28039 & 0.36194 \\
MC\_GROUSE \cite{Balzano13}   & 0.65640 & 0.89692 & 0.10308 & 0.34360 & 0.25425 & 0.31495 \\
OMoGMF \cite{Meng13,Yong18}   & \textbf{0.89943} & 0.98289 & 0.01711 & \textbf{0.10057} & 0.62033 & 0.72611 \\
RPCA\_FPCP \cite{Rodriguez13} & 0.73848 & 0.94733 & 0.05267 & 0.26152 & 0.29994 & 0.37900 \\
ST\_GRASTA \cite{He11}        & 0.45340 & 0.98205 & 0.01795 & 0.54660 & 0.44009 & 0.42367 \\
TTD\_3WD \cite{Oreifej13}     & 0.61103 & 0.97117 & 0.02883 & 0.38897 & 0.35557 & 0.40297 \\
Algorithm~\ref{alg:SCDM} (w/ Geman-McLure)                  & 0.85684 & \textbf{0.99078} & \textbf{0.00922} & 0.14316 & \textbf{0.77210} & \textbf{0.80254} \\
Algorithm~\ref{alg:SCDM} (w/ $L_1$ norm) &
0.84561 & 0.99063 & 0.00937 & 0.15439 & 0.76709 & 0.79421 \\
\bottomrule
\end{tabular}}
\end{center}
\end{table*}

\begin{table*}[tp!]
\caption{Results of our Algorithm~\ref{alg:SCDM}, compared to 3 other approaches on 6 categories of http://changedetection.net, evaluated on the 6 performance metrics of (Goyette et al. 2012).  For each pair of performance metric and category, the best result across the presented methods is highlighted in bold.}
\label{tab:DAEvsOMoGMF}
\begin{center}
{\footnotesize
\begin{tabular}{lcccccc}
\toprule
Approach and category / Performance metric & Recall & Specificity & FPR
& FNR & Precision & F1     \\[6pt]
\textbf{Algorithm~\ref{alg:SCDM} (w/ $L_1$ norm)}: & {} & {} & {} & {} & {} & {} \\
\midrule
badWeather        & 0.86589 & \textbf{0.98814} & 0.01186 & 0.13411 & 0.54689 & 0.64618 \\
baseline          & 0.84561 & \textbf{0.99063} & \textbf{0.00937} & 0.15439 & \textbf{0.76709} & \textbf{0.79421} \\
cameraJitter      & 0.59694 & \textbf{0.95928} & \textbf{0.04072} & 0.40306 & \textbf{0.55402} & \textbf{0.51324} \\
dynamicBackground & 0.46324 & \textbf{0.99677} & \textbf{0.00323} & 0.53676 & \textbf{0.65511} & \textbf{0.49254} \\
nightVideo        & \textbf{0.83646} & 0.87469 & 0.12531 & \textbf{0.16354} & 0.20992 & 0.29481 \\
shadow            & \textbf{0.76158} & \textbf{0.97612} & \textbf{0.02388} & \textbf{0.23842} & \textbf{0.64121} & \textbf{0.68493} \\
\midrule
Overall           & 0.72829 & \textbf{0.96427} & \textbf{0.03573} & 0.27171 & \textbf{0.56237} & \textbf{0.57099} \\[6pt]
\textbf{OMoGMF} \cite{Yong18}: & {} & {} & {} & {} & {} & {} \\
\midrule
badWeather        & \textbf{0.86871} & 0.98939 & 0.01061 & \textbf{0.13129} & \textbf{0.57917} & \textbf{0.67214} \\
baseline          & \textbf{0.89943} & 0.98289 & 0.01711 & \textbf{0.10057} & 0.62033 & 0.72611 \\
cameraJitter      & \textbf{0.85954} & 0.90739 & 0.09261 & \textbf{0.14046} & 0.30567 & 0.44235 \\
dynamicBackground & \textbf{0.87655} & 0.86383 & 0.13617 & \textbf{0.12345} & 0.08601 & 0.15012 \\
nightVideo        & 0.75607 & 0.92372 & 0.07628 & 0.24393 & \textbf{0.23252} & \textbf{0.31336} \\
shadow            & 0.55772 & 0.80276 & 0.03057 & 0.27562 & 0.40539 & 0.37450 \\
\midrule
Overall           & \textbf{0.80300} & 0.91166 & 0.06056 & \textbf{0.16922} & 0.37151 & 0.44643 \\[6pt]
\textbf{ST\_GRASTA} \cite{He11}: & {} & {} & {} & {} & {} & {} \\
\midrule
badWeather        & 0.26555 & 0.98971 & \textbf{0.01029} & 0.73445 & 0.45526 & 0.30498 \\
baseline          & 0.45340 & 0.98205 & 0.01795 & 0.54660 & 0.44009 & 0.42367 \\
cameraJitter      & 0.51138 & 0.91313 & 0.08687 & 0.48862 & 0.23995 & 0.31572 \\
dynamicBackground & 0.41411 & 0.94755 & 0.05245 & 0.58589 & 0.08732 & 0.13736 \\
nightVideo        & 0.42488 & \textbf{0.97224} & \textbf{0.02776} & 0.57512 & \textbf{0.24957} & 0.28154 \\
shadow            & 0.44317 & 0.96681 & 0.03319 & 0.55683 & 0.42604 & 0.41515 \\
\midrule
Overall           & 0.41875 & 0.96192 & 0.03808 & 0.58125 & 0.31637 & 0.31307 \\[6pt]
\textbf{RPCA\_FPCP} \cite{Rodriguez13}: & {} & {} & {} & {} & {} & {} \\
\midrule
badWeather        & 0.82546 & 0.84424 & 0.15576 & 0.17454 & 0.09950 & 0.16687 \\
baseline          & 0.73848 & 0.94733 & 0.05267 & 0.26152 & 0.29994 & 0.37900 \\
cameraJitter      & 0.74452 & 0.84143 & 0.15857 & 0.25548 & 0.18436 & 0.29024 \\
dynamicBackground & 0.69491 & 0.80688 & 0.19312 & 0.30509 & 0.03928 & 0.07134 \\
nightVideos       & 0.79284 & 0.85751 & 0.14249 & 0.20716 & 0.11797 & 0.19497 \\
shadow            & 0.72132 & 0.90454 & 0.09546 & 0.27868 & 0.26474 & 0.36814 \\
\midrule
Overall :         & 0.75292 & 0.86699 & 0.13301 & 0.24708 & 0.16763 & 0.24509 \\[3pt]
\bottomrule
\end{tabular}}
\end{center}
\end{table*}

We have also conducted a number of experiments 
on instances from changedetection.net \cite{goyette2012changedetection}, a benchmark often used to test low-rank approaches.
There, short videos (1,000 to 9,000 frames) 
are supplemented
with ground-truth information of what is foreground and what is background.
These experiments have been run on a single 4-core workstation (Intel Core i7-4800MQ CPU, 16~GB of RAM, RedHat~7.6/64) and results have been deposited\footnote{https://figshare.com/articles/AAAI2020\_results\_zip/10316696} in FigShare.
In Tables~\ref{tab:all-baseline} and ~\ref{tab:DAEvsOMoGMF}, we summarise the results. 
In particular, we present the false positive rate (FPR), false negative rate (FNR), specificity, precision, recall, and the geometric mean of the latter two (F1) of our method and 6 other low-rank approaches, which have been used as reference methods recently \cite{bouwmans2016handbook}.
These reference methods are implemented in \texttt{LRSLibrary} \cite{lrslibrary2015,bouwmans2015} and by the original authors of  \texttt{OMoGMF} \cite{Meng13,Yong18}, and have been used with their default settings.
Out of these, OMoGMF \cite{Yong18} is the most recent and considered to be the most robust. Still, we can improve upon the results of OMoGMF by a considerable margin: the F1 score across the 6 categories is improved by 28\% from 0.44643 to 0.57099, for example.

Further details and results are available in the appendix.
At http://changedetection.net/,
a comparison 
against four dozen other methods is readily available,
although one should like to discount methods tagged as ``supervised'', 
which are trained and tested on one and the same dataset.
A further comparison against dozens of other methods is available in \cite{8398586}.

\section{Conclusions}

We have presented a tracking result for time-varying low-rank models of time-varying matrices,
robust to both uniformly-distributed measurement noise
and arbitrarily-distributed ``sparse'' noise.
This improves upon prior work, as summarised by the recent special issues \cite{8398586,8425660}. 

Our analytical guarantees  
improve upon the state of the art in two ways.
First, we provide a bound on the tracking error in estimation of the time-varying low-rank sub-space, rather than a result restricted to the off-line case.
Second, we do not make restrictive assumptions on RIP properties, incoherence, identical covariance matrices, independence of all outlier supports, or initialisation. 
Broadly speaking, such analyses of \emph{time-varying non-convex optimisation} \cite{8442544,Tang2018ncvx,fattahi2019absence,massicot2019line},  %
seems to be an important direction for further research.

In practice, our use of randomised coordinate descent in alternating least-squares seems much
better suited to high-volume (high-dimensional, high-frequency) data streams than spectral methods
and other alternatives we are aware of. 
When the matrix $\M_k$ does not change quickly, 
performing a fixed number of iterations within an inexact step (\ref{eq:Specific}) upon arrival of a new sample makes it possible to spread the computational load over time, while still recovering a good background model.
Also, our algorithm is easy to implement and optimize. It has very few hyper-parameters, and this simplifies tuning.
Our results are hence practically relevant.

\section*{Acknowledgments}
This research received funding from the European Union
Horizon 2020 Programme (Horizon2020/2014-2020) under
grant agreement number 688380 (project VaVeL).

\bibliographystyle{aaai}
\bibliography{refs,completion,pursuit}

\begin{thebibliography}{}

\bibitem[\protect\citeauthoryear{Akhriev and Marecek}{2019}]{akhriev2019deep}
Akhriev, A., and Marecek, J.
\newblock 2019.
\newblock Deep autoencoders with value-at-risk thresholding for unsupervised
  anomaly detection.
\newblock In {\em 2019 IEEE International Symposium on Multimedia (ISM)},
  208--2083.
\newblock IEEE.

\bibitem[\protect\citeauthoryear{Akhriev, Marecek, and
  Simonetto}{2018}]{akhriev2018pursuit}
Akhriev, A.; Marecek, J.; and Simonetto, A.
\newblock 2018.
\newblock Pursuit of low-rank models of time-varying matrices robust to sparse
  and measurement noise.
\newblock {\em arXiv preprint arXiv:1809.03550}.
\newblock Full version.

\bibitem[\protect\citeauthoryear{{Balzano} and {Wright}}{2013}]{Balzano13}
{Balzano}, L., and {Wright}, S.~J.
\newblock 2013.
\newblock {On GROUSE and incremental SVD}.
\newblock In {\em 2013 5th IEEE Int. Workshop on Comp. Advances in Multi-Sensor
  Adaptive Proc. (CAMSAP)},  1--4.

\bibitem[\protect\citeauthoryear{Balzano, Chi, and Lu}{2018}]{8417980}
Balzano, L.; Chi, Y.; and Lu, Y.~M.
\newblock 2018.
\newblock Streaming {PCA} and subspace tracking: The missing data case.
\newblock {\em Proceedings of the IEEE} 106(8):1293--1310.

\bibitem[\protect\citeauthoryear{Bhojanapalli \bgroup et al\mbox.\egroup
  }{2018}]{bhojanapalli2018smoothed}
Bhojanapalli, S.; Boumal, N.; Jain, P.; and Netrapalli, P.
\newblock 2018.
\newblock Smoothed analysis for low-rank solutions to semidefinite programs in
  quadratic penalty form.
\newblock {\em Conference on Learning Theory (COLT)}.

\bibitem[\protect\citeauthoryear{Bhojanapalli, Neyshabur, and
  Srebro}{2016}]{bhojanapalli2016global}
Bhojanapalli, S.; Neyshabur, B.; and Srebro, N.
\newblock 2016.
\newblock Global optimality of local search for low rank matrix recovery.
\newblock In {\em Advances in Neural Information Processing Systems 29},
  3873--3881.

\bibitem[\protect\citeauthoryear{Boumal, Absil, and
  Cartis}{2018}]{boumal2018global}
Boumal, N.; Absil, P.-A.; and Cartis, C.
\newblock 2018.
\newblock Global rates of convergence for nonconvex optimization on manifolds.
\newblock {\em IMA Journal of Numerical Analysis} 39(1):1--33.

\bibitem[\protect\citeauthoryear{Boumal, Voroninski, and
  Bandeira}{2016}]{NIPS2016_6517}
Boumal, N.; Voroninski, V.; and Bandeira, A.
\newblock 2016.
\newblock The non-convex burer-monteiro approach works on smooth semidefinite
  programs.
\newblock In {\em Advances in Neural Information Processing Systems 29}.
\newblock  2757--2765.

\bibitem[\protect\citeauthoryear{Bouwmans, Aybat, and
  Zahzah}{2016}]{bouwmans2016handbook}
Bouwmans, T.; Aybat, N.~S.; and Zahzah, E.-h.
\newblock 2016.
\newblock {\em Handbook of robust low-rank and sparse matrix decomposition:
  Applications in image and video processing}.
\newblock Chapman and Hall/CRC.

\bibitem[\protect\citeauthoryear{Bouwmans \bgroup et al\mbox.\egroup
  }{2015}]{bouwmans2015}
Bouwmans, T.; Sobral, A.; Javed, S.; Jung, S.~K.; and Zahzah, E.-h.
\newblock 2015.
\newblock Decomposition into low-rank plus additive matrices for
  background/foreground separation: {A} review for a comparative evaluation
  with a large-scale dataset.
\newblock {\em CoRR} abs/1511.01245.

\bibitem[\protect\citeauthoryear{Cand{\`e}s and Recht}{2009}]{candes2009exact}
Cand{\`e}s, E.~J., and Recht, B.
\newblock 2009.
\newblock Exact matrix completion via convex optimization.
\newblock {\em Foundations of Computational mathematics} 9(6):717.

\bibitem[\protect\citeauthoryear{Cand{\`e}s \bgroup et al\mbox.\egroup
  }{2011}]{candes2011robust}
Cand{\`e}s, E.~J.; Li, X.; Ma, Y.; and Wright, J.
\newblock 2011.
\newblock Robust principal component analysis?
\newblock {\em Journal of the ACM} 58(3):11.

\bibitem[\protect\citeauthoryear{Cand{\`e}s}{2008}]{candes2008restricted}
Cand{\`e}s, E.~J.
\newblock 2008.
\newblock The restricted isometry property and its implications for compressed
  sensing.
\newblock {\em Comptes Rendus Mathematique} 346(9):589 -- 592.

\bibitem[\protect\citeauthoryear{Cheng and Ge}{2018}]{cheng2018non}
Cheng, Y., and Ge, R.
\newblock 2018.
\newblock Non-convex matrix completion against a semi-random adversary.
\newblock {\em arXiv preprint arXiv:1803.10846}.

\bibitem[\protect\citeauthoryear{Dutta and Li}{2017}]{dutta2017problem}
Dutta, A., and Li, X.
\newblock 2017.
\newblock On a problem of weighted low-rank approximation of matrices.
\newblock {\em SIAM Journal on Matrix Analysis and Applications}
  38(2):530--553.

\bibitem[\protect\citeauthoryear{Dutta, Li, and
  Richt{\'a}rik}{2017}]{dutta2017batch}
Dutta, A.; Li, X.; and Richt{\'a}rik, P.
\newblock 2017.
\newblock A batch-incremental video background estimation model using weighted
  low-rank approximation of matrices.
\newblock In {\em Proceedings of the IEEE International Conference on Computer
  Vision},  1835--1843.

\bibitem[\protect\citeauthoryear{Fattahi \bgroup et al\mbox.\egroup
  }{2019}]{fattahi2019absence}
Fattahi, S.; Josz, C.; Mohammadi, R.; Lavaei, J.; and Sojoudi, S.
\newblock 2019.
\newblock Absence of spurious local trajectories in time-varying optimization.
\newblock {\em arXiv preprint arXiv:1905.09937}.

\bibitem[\protect\citeauthoryear{Feng \bgroup et al\mbox.\egroup
  }{2013}]{NIPS2013_5135}
Feng, J.; Xu, H.; Mannor, S.; and Yan, S.
\newblock 2013.
\newblock Online {PCA} for contaminated data.
\newblock In Burges, C. J.~C.; Bottou, L.; Welling, M.; Ghahramani, Z.; and
  Weinberger, K.~Q., eds., {\em Advances in Neural Information Processing
  Systems 26}. Curran Associates, Inc.
\newblock  764--772.

\bibitem[\protect\citeauthoryear{Goyette \bgroup et al\mbox.\egroup
  }{2012}]{goyette2012changedetection}
Goyette, N.; Jodoin, P.-M.; Porikli, F.; Konrad, J.; and Ishwar, P.
\newblock 2012.
\newblock Changedetection. net: A new change detection benchmark dataset.
\newblock In {\em Computer Vision and Pattern Recognition Workshops (CVPRW),
  2012 IEEE Computer Society Conference on},  1--8.
\newblock IEEE.

\bibitem[\protect\citeauthoryear{Guo, Qiu, and Vaswani}{2014}]{guo2014online}
Guo, H.; Qiu, C.; and Vaswani, N.
\newblock 2014.
\newblock An online algorithm for separating sparse and low-dimensional signal
  sequences from their sum.
\newblock {\em IEEE Transactions on Signal Processing} 62(16):4284--4297.

\bibitem[\protect\citeauthoryear{{He}, {Balzano}, and {Lui}}{2011}]{He11}
{He}, J.; {Balzano}, L.; and {Lui}, J. C.~S.
\newblock 2011.
\newblock {Online Robust Subspace Tracking from Partial Information}.
\newblock {\em arXiv e-prints}  arXiv:1109.3827.

\bibitem[\protect\citeauthoryear{Huber}{1981}]{huber1981robust}
Huber, P.~J.
\newblock 1981.
\newblock {\em Robust Statistics}.
\newblock Wiley-Interscience.

\bibitem[\protect\citeauthoryear{Jain and Kar}{2017}]{jain2017non}
Jain, P., and Kar, P.
\newblock 2017.
\newblock {\em Non-convex optimization for machine learning}, volume~10 of {\em
  Foundations and Trends{\textregistered} in Machine Learning}.
\newblock Now Publishers.

\bibitem[\protect\citeauthoryear{Lerman and Maunu}{2018}]{8425657}
Lerman, G., and Maunu, T.
\newblock 2018.
\newblock An overview of robust subspace recovery.
\newblock {\em Proceedings of the IEEE} 106(8):1380--1410.

\bibitem[\protect\citeauthoryear{Lin, Liu, and Su}{2011}]{Lin11}
Lin, Z.; Liu, R.; and Su, Z.
\newblock 2011.
\newblock Linearized alternating direction method with adaptive penalty for
  low-rank representation.
\newblock In Shawe-Taylor, J.; Zemel, R.~S.; Bartlett, P.~L.; Pereira, F.; and
  Weinberger, K.~Q., eds., {\em Advances in Neural Information Processing
  Systems 24}.
\newblock  612--620.

\bibitem[\protect\citeauthoryear{Liu \bgroup et al\mbox.\egroup
  }{2013}]{6180173}
Liu, G.; Lin, Z.; Yan, S.; Sun, J.; Yu, Y.; and Ma, Y.
\newblock 2013.
\newblock Robust recovery of subspace structures by low-rank representation.
\newblock {\em IEEE Transactions on Pattern Analysis and Machine Intelligence}
  35(1):171--184.

\bibitem[\protect\citeauthoryear{{Liu} \bgroup et al\mbox.\egroup
  }{2018}]{8442544}
{Liu}, J.; {Marecek}, J.; {Simonetta}, A.; and {Takac}, M.
\newblock 2018.
\newblock A coordinate-descent algorithm for tracking solutions in time-varying
  optimal power flows.
\newblock In {\em 2018 Power Systems Computation Conference (PSCC)},  1--7.

\bibitem[\protect\citeauthoryear{Ma and Aybat}{2018}]{8412568}
Ma, S., and Aybat, N.~S.
\newblock 2018.
\newblock Efficient optimization algorithms for robust principal component
  analysis and its variants.
\newblock {\em Proceedings of the IEEE} 106(8):1411--1426.

\bibitem[\protect\citeauthoryear{Malistov}{2014}]{Malistov2014}
Malistov, A.
\newblock 2014.
\newblock Estimation of background noise in traffic conditions and selection of
  a threshold for selecting mobile objects.
\newblock {\em Actual issues of modern science} 4.
\newblock In Russian with English abstract in pp.~12-13 at
  \url{http://www.malistov.ru/docs/dissertation/abstract_malistov.pdf}.

\bibitem[\protect\citeauthoryear{Mardani, Mateos, and
  Giannakis}{2013}]{Mardani2013}
Mardani, M.; Mateos, G.; and Giannakis, G.~B.
\newblock 2013.
\newblock Dynamic anomalography: Tracking network anomalies via sparsity and
  low rank.
\newblock {\em IEEE Journal of Selected Topics in Signal Processing}
  7(1):50--66.

\bibitem[\protect\citeauthoryear{Marecek \bgroup et al\mbox.\egroup
  }{2018}]{marecek2018low}
Marecek, J.; Maroulis, S.; Kalogeraki, V.; and Gunopulos, D.
\newblock 2018.
\newblock Low-rank methods in event detection.
\newblock {\em arXiv preprint arXiv:1802.03649}.

\bibitem[\protect\citeauthoryear{Marecek, Richtarik, and
  Takac}{2017}]{marecek2017matrix}
Marecek, J.; Richtarik, P.; and Takac, M.
\newblock 2017.
\newblock Matrix completion under interval uncertainty.
\newblock {\em European Journal of Operational Research} 256(1):35--43.

\bibitem[\protect\citeauthoryear{Massicot and Marecek}{2019}]{massicot2019line}
Massicot, O., and Marecek, J.
\newblock 2019.
\newblock On-line non-convex constrained optimization.
\newblock {\em arXiv preprint arXiv:1909.07492}.

\bibitem[\protect\citeauthoryear{Meng and Torre}{2013}]{Meng13}
Meng, D., and Torre, F. D.~L.
\newblock 2013.
\newblock Robust matrix factorization with unknown noise.
\newblock In {\em Proceedings of the 2013 IEEE International Conference on
  Computer Vision}, ICCV '13,  1337--1344.
\newblock Washington, DC, USA: IEEE Computer Society.

\bibitem[\protect\citeauthoryear{Nesterov}{2012}]{Nesterov2012}
Nesterov, Y.
\newblock 2012.
\newblock Efficiency of coordinate descent methods on huge-scale optimization
  problems.
\newblock {\em SIAM Journal on Optimization} 22(2):341--362.

\bibitem[\protect\citeauthoryear{{Oreifej}, {Li}, and {Shah}}{2013}]{Oreifej13}
{Oreifej}, O.; {Li}, X.; and {Shah}, M.
\newblock 2013.
\newblock Simultaneous video stabilization and moving object detection in
  turbulence.
\newblock {\em IEEE Transactions on Pattern Analysis and Machine Intelligence}
  35(2):450--462.

\bibitem[\protect\citeauthoryear{Park \bgroup et al\mbox.\egroup
  }{2017}]{Carmanis2017}
Park, D.; Kyrillidis, A.; Carmanis, C.; and Sanghavi, S.
\newblock 2017.
\newblock {Non-square matrix sensing without spurious local minima via the
  Burer-Monteiro approach}.
\newblock In {\em Proceedings of the 20th International Conference on
  Artificial Intelligence and Statistics},  65--74.

\bibitem[\protect\citeauthoryear{{Rodriguez} and
  {Wohlberg}}{2013}]{Rodriguez13}
{Rodriguez}, P., and {Wohlberg}, B.
\newblock 2013.
\newblock Fast principal component pursuit via alternating minimization.
\newblock In {\em 2013 IEEE International Conference on Image Processing},
  69--73.

\bibitem[\protect\citeauthoryear{Sawhney and Ayer}{1996}]{Sawhney96}
Sawhney, H.~S., and Ayer, S.
\newblock 1996.
\newblock Compact representations of videos through dominant and multiple
  motion estimation.
\newblock {\em IEEE Transactions on Pattern Analysis and Machine Intelligence}
  18(8):814--830.

\bibitem[\protect\citeauthoryear{Sobral, Bouwmans, and
  Zahzah}{2015}]{lrslibrary2015}
Sobral, A.; Bouwmans, T.; and Zahzah, E.-h.
\newblock 2015.
\newblock Lrslibrary: Low-rank and sparse tools for background modeling and
  subtraction in videos.
\newblock In {\em Robust Low-Rank and Sparse Matrix Decomposition: Applications
  in Image and Video Processing}. CRC Press, Taylor and Francis Group.

\bibitem[\protect\citeauthoryear{Tang \bgroup et al\mbox.\egroup
  }{2018}]{Tang2018ncvx}
Tang, Y.; {Dall'Anese}, E.; Bernstein, A.; and Low, S.
\newblock 2018.
\newblock Running primal-dual gradient method for time-varying nonconvex
  problems.
\newblock {\em arXiv preprint arXiv:1812.00613}.

\bibitem[\protect\citeauthoryear{Vaswani and Narayanamurthy}{2018}]{8425658}
Vaswani, N., and Narayanamurthy, P.
\newblock 2018.
\newblock Static and dynamic robust {PCA} and matrix completion: A review.
\newblock {\em Proceedings of the IEEE} 106(8):1359--1379.

\bibitem[\protect\citeauthoryear{Vaswani \bgroup et al\mbox.\egroup
  }{2018}]{8398586}
Vaswani, N.; Bouwmans, T.; Javed, S.; and Narayanamurthy, P.
\newblock 2018.
\newblock Robust subspace learning: Robust {PCA}, robust subspace tracking, and
  robust subspace recovery.
\newblock {\em IEEE Signal Processing Magazine} 35(4):32--55.

\bibitem[\protect\citeauthoryear{{Vaswani}, {Chi}, and
  {Bouwmans}}{2018}]{8425660}
{Vaswani}, N.; {Chi}, Y.; and {Bouwmans}, T.
\newblock 2018.
\newblock Rethinking {PCA} for modern data sets: Theory, algorithms, and
  applications [scanning the issue].
\newblock {\em Proceedings of the IEEE} 106(8):1274--1276.

\bibitem[\protect\citeauthoryear{Wang, Simoncelli, and Bovik}{2003}]{wang2003}
Wang, Z.; Simoncelli, E.~P.; and Bovik, A.~C.
\newblock 2003.
\newblock Multiscale structural similarity for image quality assessment.
\newblock In {\em Signals, Systems and Computers, 2004. Conference Record of
  the Thirty-Seventh Asilomar Conference on}, volume~2,  1398--1402.
\newblock Ieee.

\bibitem[\protect\citeauthoryear{{Yong} \bgroup et al\mbox.\egroup
  }{2018}]{Yong18}
{Yong}, H.; {Meng}, D.; {Zuo}, W.; and {Zhang}, L.
\newblock 2018.
\newblock Robust online matrix factorization for dynamic background
  subtraction.
\newblock {\em IEEE Transactions on Pattern Analysis and Machine Intelligence}
  40(7):1726--1740.

\end{thebibliography}

\clearpage
\onecolumn

\section{Proofs}
\label{appA}

\subsection{Properties of the Problem}
\label{sec:analysis1}
    
First, let us see that while $f$ is not convex in both $\L$ and $\R$, it is convex in either $\L$ or $\R$.
    Jain \cite{jain2017non} calls this property \emph{marginal convexity}: A function $f(\L, \R)$ is marginally convex in $\L$, if for every value of $\R \in \Reals^{r\times n}$, the function $f(\cdot,\R): \Reals^{m\times r} \rightarrow \Reals$ is convex.
    
\begin{lemma}[Marginal Convexity]
\label{lem:marg-cvx-fn}
As continuously differentiable function,  $f: \Reals^{m\times r} \times \Reals^{r\times n} \rightarrow \Reals$ is \emph{marginally convex} i.e., for every $\L', \L'' \in \Reals^{m\times r}$, we have
\[
f(\L'', \R) \geq f( \L', \R) + \ip{\nabla_x f( \L', \R)}{ \L'' - \L'},
\]
where $\nabla_x f( \L', \R)$ is the partial gradient of $f$ with respect to its first variable at the point $( \L', \R)$, and likewise for $\R$.
\end{lemma}

\begin{proof}
By simple calculus.
\end{proof}

Next, let us extend the reasoning of Marecek et al. \cite{marecek2017matrix} to further properties of the 
function restricted to only $\L$ or only $\R$. 
Jain \cite[Section 4.4]{jain2017non} calls a continuously differentiable function $f: \Reals^{m\times r} \times \Reals^{r\times n} \rightarrow \Reals$  (uniformly) $\alpha$-marginally strongly convex (MSC) in $\L$ if for all $\R$, the function $f(\L, \R)$ is $\alpha$ strongly convex for the constant $\R$.
Likewise for (uniformly) $\beta$-marginally strongly smooth (MSS) functions.
The textbook example \cite[Figure 4.1]{jain2017non} is $f: \Reals \times \Reals \rightarrow \Reals$, $f(x, y) = x \cdot y$.
Notice the similarity to the Restricted Isometry Property (RIP) of \cite{candes2008restricted}.

\begin{lemma}[MSC/MSS]
\label{lem:marg-strong-cvx-smooth-fn}
There are finite $\alpha, \beta$, such that 
the function $f(\cdot, \R): \Reals^{m\times r} \rightarrow \Reals$ is $\alpha$-strongly convex and $\beta$-strongly smooth,
i.e., 
for every value of $\R \in \Reals^{r\times n}$, 
for every $\L',\L' \in \Reals^{m\times r}$, we have
\begin{align*}
\frac{\alpha}{2}\norm{ \L'' - \L'}_2^2 & \leq f( \L'', \R) - f( \L', \R) - \ip{g}{ \L'' - \L'} \\ & \leq \frac{\beta}{2}\norm{ \L'' - \L'}_2^2,
\end{align*}
where $g = \nabla_x f( \L', \R)$ is the partial gradient of $f$ with respect to its first variable at the point $(\L',\R)$. Likewise, the function $f(\L, \cdot): \Reals^{n\times r} \rightarrow \Reals$ is $\alpha'$-strongly convex and $\beta'$-strongly smooth.
\end{lemma}

\begin{proof}[Proof of Lemma\ref{lem:marg-strong-cvx-smooth-fn} ]
Notice that $W_{i\hat{r}}$, the modulus of Lischitz continuity of the gradient of $f$ restricted to the $\L_{i \hat r}$
	sampled is:
\begin{align}
W_{k, i\hat{r}} = \mu + \sum_{(j, v)} \R^2_{k, \hat r j}
\end{align}
where the superscript denotes squaring, rather than an iteration index, which we omit for brevity.
Considering the level set is bounded, $W_{k,j\hat{r}}$ is bounded and we have the result.
Similarly $V_{ \hat{r} j}$, the modulus of Lischitz continuity of the gradient of $f$ restricted to the $\R_{ k, i \hat r }$ is:
\begin{align}
V_{k, \hat{r} j} = \mu +  \sum_{(i, v)} \L^2_{k, i \hat r},
\end{align}
where again, the superscript denotes squaring, rather than an iteration index.
\end{proof}


Next, let us consider some more definitions of \cite{jain2017non}. For any $\R$, we say that $\tilde \L$ is a marginally optimal coordinate with respect to $\R$, and use the shorthand $\tilde \L \in \mopt(\R)$, if $f(\tilde \L, \R) \leq f( \L, \R)$ for all $\L$. Similarly for any $\L$,  $\tilde \R \in \mopt(\L)$ if $\tilde \R$ is a marginally optimal coordinate with respect to $\L$. Then:

\begin{definition}[Bistable Point of \cite{jain2017non}]
\label{defn:bistable}
Given a function $f$ over two variables constrained within the sets $\cX,\cY$ respectively, a point $( \L, \R) \in \cX\times\cY$ is considered a \emph{bistable} point if $y \in \mopt( \L)$ and $x \in \mopt(y)$ i.e., both coordinates are marginally optimal with respect to each other.
\end{definition}

\begin{lemma}[Jain et al. \cite{jain2017non}]
A point $( \L, \R)$ is bistable with respect to a continuously differentiable function 
$f: \Reals^{m\times r}  \times \Reals^{r\times n}$ that is marginally convex in both its variables if 
and only if  $\nabla f( \L, \R) = \vzero$.
\end{lemma}

\begin{proof}[Proof of Lemma \ref{lem:robustbistability}]
Notice that each element of the matrix is bounded both from above and from below. The level sets are hence bounded, whereby we obtain the result.
\end{proof}


Then, we can restate Theorem 1 of \cite{marecek2017matrix}:

\begin{theorem}[Based on Theorem 1 in Marecek et al. \cite{marecek2017matrix}]
\label{lemma:bistable-stat}
For any $\overline \tau > 0$ and  $\hatSr, \hatSc$ sampled uniformly at random,
the limit point 
$\liminf_{T \to \infty} ( \L^{T,\overline \tau}_k, \R^{T,\overline \tau}_k)$ of Algorithm~\ref{alg:SCDM} 
is bistable with probability 1.
\end{theorem}

The proof follows that of Theorem 1 in \cite{marecek2017matrix}. There, however, the analysis of \cite{marecek2017matrix} ends.

	\subsection{The Limit Point}
	\label{sec:limitPoint}

Next, consider further properties of the limit point under the assumptions above. 
To do so, we present some more definitions of Jain  \cite{jain2017non}:  

\begin{definition}[Robust Bistability Property of \cite{jain2017non}]
\label{defn:rob-bistable}
A function $f: \Reals^{m\times r}  \times \Reals^{r\times n} \rightarrow \Reals$ satisfies the $C$-robust bistability property if for some $C > 0$, for every $( \L, \R) \in \Reals^{m\times r}  \times \Reals^{r\times n}$, $\tilde \R \in \mopt( \L)$ and $\tilde \L \in \mopt(\R)$, we have
\begin{align}
f( \L, \R^\ast) + f( \L^\ast, \R) - 2f^\ast \leq C\cdot \left( 2f( \L, \R) - f( \L,\tilde \R) - f(\tilde \L, \R) \right).
\end{align}
\end{definition}

Subsequently:

\begin{lemma}
\label{lem:robustbistability}
Under Assumption \ref{ass:partition}, there exists a finite $C > 0$, such that the function $f$ \eqref{eq:Specific} satisfies the $C$-robust bistability property.
\end{lemma}

Much more detailed results, bounding the constant $C$, are available in many regimes, e.g., 
when each element of the matrix is sampled with a probability larger than a certain instance-specific $p$
from a certain ensemble \cite{cheng2018non},
and more generally when one allows from a certain smoothing \cite{bhojanapalli2016global,bhojanapalli2018smoothed}.
Further, one can use the results of \cite{Carmanis2017} to prove its satisfaction under the 
Restricted Isometry Property (RIP) of \cite{candes2008restricted}.

Next, let us state a technical lemma:


\begin{lemma}[Based on Lemma 4.4 of \cite{jain2017non}]
\label{lemma:local-conv-gam}
Under Assumption \ref{ass:partition},
for any $( \L, \R) \in \Reals^{m\times r}  \times \Reals^{r\times n}$, $\tilde \R \in \mopt( \L)$ 
and $\tilde \L \in \mopt(\R)$,
\begin{align}
\norm{ \L- \L^\ast}_2^2 + \norm{\R-\R^\ast}_2^2 \leq \frac{C\beta}{\alpha} \left( \norm{ \L-\tilde \L}_2^2 + \norm{\R -\tilde \R}_2^2 \right)
\end{align}
\end{lemma}

\begin{proof}[Proof of Lemma \ref{lemma:local-conv-gam}]
Notice that $f$ is
$\alpha$-MSC, $\beta$-MSS in both $\L$ and $\R$, as shown in Lemma \ref{lem:marg-cvx-fn} and 
\ref{lem:marg-strong-cvx-smooth-fn}.
From Lemma \ref{lem:marg-strong-cvx-smooth-fn}:
\begin{align}
f( \L, \R^\ast) + f( \L^\ast, \R) \geq 2f^\ast + \frac{\alpha}{2}\left( \norm{ \L- \L^\ast}_2^2 + \norm{y- \L^\ast}_2^2 \right)\\
2f( \L, \R) \leq f( \L,\tilde \R) + f(\tilde \L, \R) + \frac{\beta}{2}\left( \norm{ \L-\tilde \L}_2^2+\norm{\R-\tilde \R}_2^2 \right)
\end{align}
Applying robust bistability of Lemma \ref{lem:robustbistability} then proves the result.
\end{proof}

Using Lemma \ref{lemma:local-conv-gam}, we can present a bound on the limit point and the rate of convergence to it, i.e., prove Theorem \ref{thm:gam-msc-mss-proof}, which we restate here for convenience:

\begin{theorem}[]
There exists $\overline \tau > 0$, such that
Algorithm~\ref{alg:SCDM} with the initialization to all-zero vector after 
at most $T = O(\log\frac{1}{\epsilon})$ steps has $f( \L^T, \R^T) \leq f^\ast + \epsilon$
with probability 1.
\end{theorem}

\begin{proof}
We follow \cite{jain2017non} and use $\Phi^{(k)} = f( \L^{(k)}, \R^{(k)}) - f^\ast$ as the potential function. 
The $\overline \tau$ we require depends on the cardinality of $\hatSr, \hatSc$, and the model of computation, but should be large enough for marginal optimisation, i.e., $\nabla_ \L f( \L^\ast, \R^\ast) = 0$.
Then, Lemma \ref{lem:marg-strong-cvx-smooth-fn} assures:
\[
f( \L^{(k+1)}, \R^\ast) - f( \L^\ast, \R^\ast) \leq \frac{\beta}{2}\norm{ \L^{(k+1)}- \L^\ast}_2^2.
\]
Further, considering $ \R^{(k+1)} \in \mopt( \L^{(k+1)})$, we have:
\begin{align*}
\Phi^{(k + 1)} & = f( \L^{(k+1)}, \R^{(k+1)}) - f^\ast \leq f( \L^{(k+1)}, \R^\ast) - f^\ast 
 \leq \frac{\beta}{2}\norm{ \L^{(k+1)}- \L^\ast}_2^2,
\end{align*}
Again, considering $\nabla_ \L f( \L^{(k+1)}, \R^{(k)}) = 0$ for large-enough $\hatSr$,
\begin{align*}
f( \L^{(k)}, \R^{(k)}) &\geq f( \L^{(k+1)}, \R^{(k)}) + \frac{\alpha}{2}\norm{ \L^{(k+1)}- \L^{(k)}}_2^2\\
&\geq f( \L^{(k+1)}, \R^{(k+1)}) + \frac{\alpha}{2}\norm{ \L^{(k+1)}- \L^{(k)}}_2^2,
\end{align*}
and consequently
\[
\Phi^{(k)} - \Phi^{(k + 1)} \geq \frac{\alpha}{2}\norm{ \L^{(k+1)}- \L^{(k)}}_2^2.
\]
Applying Lemma~\ref{lemma:local-conv-gam},
\begin{align*}
\norm{ \L^{(k)}- \L^\ast}_2^2 & \leq \norm{ \L^{(k)}- \L^\ast}_2^2 + \norm{ \R^{(k)}- \R^\ast}_2^2\\
& \leq \frac{C\beta}{\alpha}\norm{ \L^{(k)}- \L^{(k+1)}}_2^2.
\end{align*}
Using $(a+b)^2 \leq 2(a^2+b^2)$,
\begin{align*}
\Phi^{(k + 1)} &\leq \frac{\beta}{2}\norm{ \L^{(k+1)}- \L^\ast}_2^2 \\
& \leq \beta \left( \norm{ \L^{(k+1)}- \L^{(k)}}_2^2 + \norm{ \L^{(k)}- \L^\ast}_2^2 \right) \\
&\leq \beta(1 + C\kappa)\norm{ \L^{(k+1)}- \L^{(k)}}_2^2 \\
& \leq 2\kappa(1 + C\kappa) \left( \Phi^{(k)} - \Phi^{(k + 1)} \right).
\end{align*}
Finally, by simple algebra,
\begin{align}
\Phi^{(k + 1)} \leq \eta_0\cdot\Phi^{(k)},
\end{align}
where 
\begin{align}
\label{eta0}
\eta_0 = \frac{2\kappa(1 + C\kappa)}{1+2\kappa(1 + C\kappa)} < 1.
\end{align}
\end{proof}


Finally:

\begin{proof}[Proof of Theorem \ref{theorem.inexact}]
The proof follows from Theorem~\ref{thm:gam-msc-mss-proof}, by invoking the triangle inequality and 
the sum of a geometric series. In particular, due to Theorem~\ref{thm:gam-msc-mss-proof}, one has for each $k$
\begin{align}
    f(\L_k, \R_k; \M_k) - f(\L_{k}^*, \R_{k}^*; \M_{k}) \leq \eta_0 (f(\L_{k-1}, \R_{k-1}; \M_{k}) - f(\L_{k}^*, \R_{k}^*; \M_{k})).
\end{align}
By summing and subtracting $\eta_0 f(\L_{k-1}, \R_{k-1}; \M_{k-1})$ to the right-hand-side and putting without loss of generality $f(\L_k^*, \R_k^*; \M_k) = f(\L_k^*, \R_k^*; \M_{k-1})$,
\begin{align}
    f(\L_k, \R_k; \M_k) - f(\L_{k}^*, \R_{k}^*; \M_{k}) \leq & \eta_0 (f(\L_{k-1}, \R_{k-1}; \M_{k}) - f(\L_{k-1}, \R_{k-1}; \M_{k-1})  + \notag \\ & + f(\L_{k-1}, \R_{k-1}; \M_{k-1})- f(\L_{k-1}^*, \R_{k-1}^*; \M_{k-1})), 
\end{align}
and by using Assumption~\ref{as:varying}
\begin{align}
    f(\L_k, \R_k; \M_k) - f(\L_{k}^*, \R_{k}^*; \M_{k}) \leq \eta_0 (f(\L_{k-1}, \R_{k-1}; \M_{k-1}) - f(\L_{k-1}^*, \R_{k-1}^*; \M_{k-1})) + \eta_0 e.
\end{align}
By summation of geometric series, the claim is proven.
\end{proof}

\clearpage
\section{Details of the Thresholding} 
\label{app:threshold}

As suggested in the main body of the text, we start by looking for the best linear combination ${\bf c}$ that minimizes difference in $L_1$:
\begin{equation}
{\bf c}_{opt} =
\arg \min_{{\bf c}} \|{\bf c}{\bf R} - {\bf f}\|_1 =
\arg \min_{{\bf c}} \sum_{i=1}^N |({\bf c}{\bf R})_i - {\bf f}_i|,
\label{eq:coptL1}
\end{equation}
where ${\bf c}$ is a $1{\times}$rank vector, ${\bf f}$ is a 2D image flattened into $1{\times}N$ vector, and $({\bf c}{\bf R})_i$ is the scalar result of multiplication between vector ${\bf c}$ and $i$-th column of matrix ${\bf R}$.
Due to the robust property of $L_1$ norm, the formulation (\ref{eq:coptL1}) provides a close approximation of the new frame at the majority of stationary (background) points, while leaving residuals at the ``moving'' (foreground) points relatively high.
By introducing the additional variables $m_i$: $|({\bf c}{\bf R})_i - {\bf f}_i| \le m_i$, for all $i=\overline{1,N}$, the optimization problem can be reformulated as a linear program:
\begin{equation}
\begin{array}{cccccc}
\mbox{minimize:}   &    {}   & {} & \sum_{i=1}^N m_i & {}  & {}   \\
{} \\
\mbox{subject to:} & 0 & \le & m_i & < & +\infty, \\
{}                 & -\infty & <  & ({\bf c}{\bf R})_i - m_i & \le & f_i, \\
{}                 & f_i & \le & ({\bf c}{\bf R})_i + m_i & < & +\infty.
\end{array}
\end{equation}
Alternatively, one can consider the robust Geman-McLure function $\rho(r,\sigma) = {r^2}/(r^2 + \sigma^2)$ as featured in \cite{Sawhney96}, where parameter $\sigma$ is estimated from the distribution of residuals over the course of optimization
\begin{equation}
{\bf c}_{opt} = \arg\,\min_{{\bf c}}\,\sum_{i=1}^N \rho\left(({\bf c}{\bf R})_i - f_i\right).
\label{eq:copt-robust}
\end{equation}
In practice, both (\ref{eq:coptL1}) and (\ref{eq:copt-robust}) produce results of similar quality, with a slightly better statistical performance of (\ref{eq:coptL1}) at a minor additional expense in terms of run-time, compared to the use of gradient methods \cite{Sawhney96} in minimisation of (\ref{eq:copt-robust}).


After the optimal linear combination ${\bf c}_{opt}$ has been obtained in (\ref{eq:coptL1}), the next step is to compute residuals $r_i$ = $|({\bf c}{\bf R})_i - {\bf f}_i|$ and threshold them into those generated by the low-rank model, $r_i < T$, and the remainder, $r_i >= T$, where $T$ is some threshold. Thresholding for background subtraction is a vast area by itself. 
Although locally adapted threshold may work best, it is quite common to choose a single threshold for each frame. We follow the same practice:
As often \cite{Malistov2014,akhriev2019deep} in Computer Vision, 
we seek a threshold of the \textit{highest sensitivity}, when isolated points ``just'' show up.
In particular, we seek a threshold such that a certain fraction (0.0025) of $3{\times}3$ contiguous patches have 1 or 2 pixels exceeding the threshold, as suggested in Figure \ref{fig:configs}. 
To explain this in detail, consider the RGB colour images, where the point-wise 2D \textit{residual map} is computed as follows:
$$
r_{i} = \left|R^{(f)}_i - R^{(b)}_i\right| +
        \left|G^{(f)}_i - G^{(b)}_i\right| +
        \left|B^{(f)}_i - B^{(b)}_i\right|,
$$
where subscripts $f$ and $b$ stands for current frame and background respectively, and index $i$ enumerates image pixels. Other metrics like Euclidean one are also possible. We accumulate so called histogram of thresholds by analysing $3 \times 3$ neighbourhood of each point in the \textit{residual map}. There are several how residual value at the central point of relates to its neighbour.

Let us consider one example. Suppose, the central value in the largest one $v_1$ and we pick up the second $v_2$ and the third $v_3$ largest ones from the $3 \times 3$ vicinity, $v_3 \le v_2 \le v_1$, and all the values are integral as usual for images. If a threshold happens in the interval $[v_3+1 \ldots v_1]$ then one of the patterns depicted on Figure~\ref{fig:configs} will show up after thresholding. As such, this particular point ``votes'' for the range $[v_3+1 \ldots v_1]$ in the histogram of thresholds, which means we increment counters in the bins $v_3+1$ to $v_1$. Repeating the process for all the points, we arrive to the histogram of thresholds as shown in Figure~\ref{fig:histthr}.
The region around the mode of the histogram ($50\%$ of its area), outlined by yellow margins on Figure~\ref{fig:histthr}, mostly contains noise. We start search for the optimum threshold from the right margin to the right until the value of histogram bin is less then $0.0025{\cdot}N$, where $N$ is the number of pixels. We found experimentally that the fraction $0.0025$ works the best, although its value can be varied without drastic effect.

For another example of use of similar thresholding techniques, please see \cite{akhriev2019deep}.

\begin{figure}[hb]
\begin{center}
\includegraphics[width=0.7\textwidth]{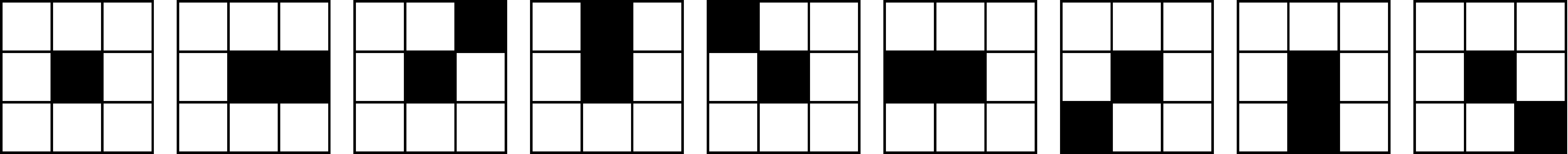}
\caption{The configurations of the $3 \times 3$ contiguous patches, whose fraction within all the $3 \times 3$ contiguous patches is sought.}
\label{fig:configs}
\end{center}
\end{figure}

\begin{figure}[ht]
\begin{center}
\includegraphics[width=0.7\textwidth]{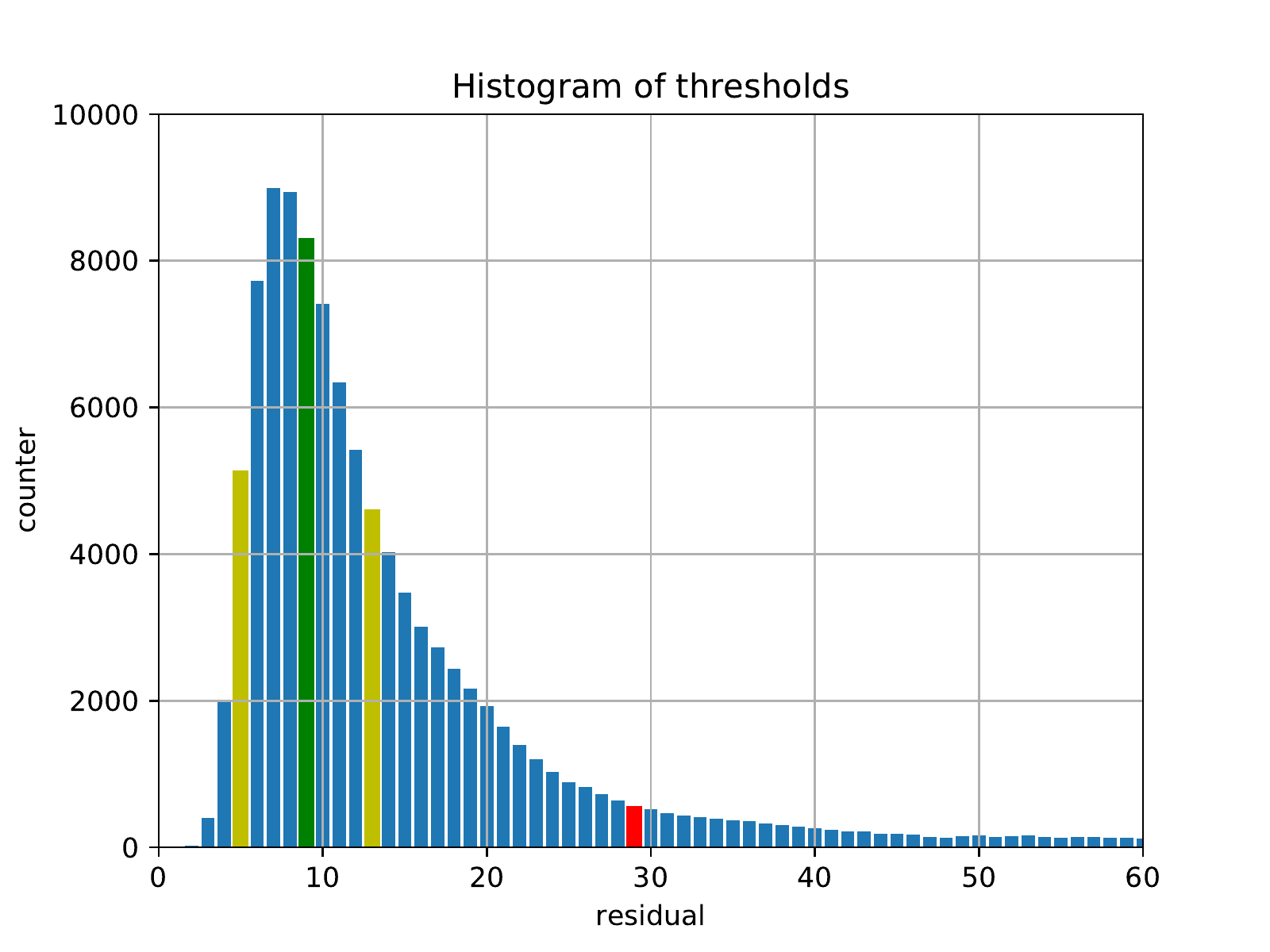}
\end{center}
\caption{A histogram of residuals. The histogram was truncated from the original $3{\cdot}255$ residuals to allow for some clarity of presentation. In green, there is the middle of the least-width interval representing half of the mass. In yellow, there are the end-points of the interval. In red, the ``optimal'' threshold we use.}
\label{fig:histthr}
\end{figure}




\clearpage
\section{A Derivation of the Step Size}
\label{appC}

\newcommand{\Xlo}{\underline{\M}_{k, ij}}
\newcommand{\Xup}{\overline{\M}_{k, ij}}
\newcommand{\Lset}{{\cal L}}
\newcommand{\Uset}{{\cal U}}
\newcommand{\Lir}{\L_{k, i r}}
\newcommand{\Rrj}{\R_{k, r j}}

\subsection{Minimisation of the objective function in $\Lir$}

\begin{eqnarray}
f(\L_k, \R_k) &=& \frac{\mu}{2}\sum_{i,j}
\left(\L_{k,ij}^2 + \R_{k,ij}^2\right) + \nonumber \\
{} & {} & \frac{1}{2} \sum_{\L_{k,i:} \R_{k,:j} < \Xlo}
\left(\L_{i:} \R_{:j} - \Xlo\right)^2 + \nonumber \\
{} & {} & \frac{1}{2} \sum_{\L_{k,i:} \R_{k,:j} > \Xup}
\left(\L_{k,i:} \R_{k,:j} - \Xup\right)^2.
\end{eqnarray}

\begin{eqnarray}
\frac{\partial f}{\partial\Lir} &=& \mu\Lir + \nonumber \\
{} & {} & \sum_{j \,:\, \L_{k,i:} \R_{k,:j} < \Xlo}
\left(\L_{k,i:} \R_{k,:j} - \Xlo\right)\Rrj + \nonumber \\
{} & {} & \sum_{j \,:\, \L_{k,i:} \R_{k,:j} > \Xup}
\left(\L_{k,i:} \R_{k,:j} - \Xup\right)\Rrj.
\end{eqnarray}

\begin{equation}
W_{i r} \triangleq
 \mu +
\sum_{j \,:\, \L_{k,i:} \R_{k,:j} < \Xlo} \Rrj^2 +
\sum_{j \,:\, \L_{k,i:} \R_{k,:j} > \Xup} \Rrj^2.
\end{equation}

\begin{equation}
\delta = -\frac{\partial   f}{\partial\Lir} \Big/
          W_{k,i r}.
\end{equation}

\begin{equation}
\Lir \leftarrow \Lir + \delta.
\end{equation}

\begin{equation}
\A_{k,i j} \leftarrow \A_{k,i j} + \delta\Rrj \qquad \forall j.
\end{equation}

\vfill 
\pagebreak
\subsection{Minimisation of the objective function in $\Rrj$}

\begin{eqnarray}
f({\L_k}, {\R_k}) &=& \frac{\mu}{2}\sum_{i,j}
\left(\L_{k,ij}^2 + \R_{k,ij}^2\right) + \nonumber \\
{} & {} & \frac{1}{2} \sum_{\L_{k,i:} \R_{k,:j} < \Xlo}
\left(\L_{k,i:} \R_{k,:j} - \Xlo\right)^2 + \nonumber \\
{} & {} & \frac{1}{2} \sum_{\L_{k,i:} \R_{k,:j} > \Xup}
\left(\L_{k,i:} \R_{k,:j} - \Xup\right)^2.
\end{eqnarray}

\begin{eqnarray}
\frac{\partial f}{\partial\Rrj} &=& \mu\Rrj + \nonumber \\
{} & {} & \sum_{i \,:\, \L_{k,i:} \R_{k,:j} < \Xlo}
\left(\L_{k,i:} \R_{k,:j} - \Xlo\right)\Lir + \nonumber \\
{} & {} & \sum_{i \,:\, \L_{k,i:} \R_{k,:j} > \Xup}
\left(\L_{k,i:} \R_{k,:j} - \Xup\right)\Lir.
\end{eqnarray}

\begin{equation}
V_{k,r j} \triangleq
 \mu +
\sum_{i \,:\, \L_{k,i:} \R_{k,:j} < \Xlo} \Lir^2 +
\sum_{i \,:\, \L_{k,i:} \R_{k,:j} > \Xup} \Lir^2.
\end{equation}

\begin{equation}
\delta = -\frac{\partial   f}{\partial\Rrj} \Big/
          V_{k,r j}.
\end{equation}

\begin{equation}
\Rrj \leftarrow \Rrj + \delta.
\end{equation}

\begin{equation}
\A_{k,i j} \leftarrow \A_{k,i j} + \delta\Lir \qquad \forall i.
\end{equation}

\clearpage 
\section{Complete Results}
\label{sec:resultsComplete}

In Table \ref{tab:scores}, we present the overall results on changedetection.net as the average over all the frames of a video, 
with a standard deviation 
 in parentheses.
First, we present MS-SSIM of \cite{wang2003}, a well-known measure of similarity of the background of each frame to our rank-4 estimate thereof,
which is also known as the multiscale structural similarity for image quality.
There, our estimates perform rather well, with the exception of videos featuring dynamic backgrounds such as waves and reflections of sun light on water,
where the low-rank model is not updated often enough to capture all of the rapid changes.
Next, we present the F1 score, which is the harmonic mean of precision and recall
 and which we used the code provided by CDnet to evaluate against the ground truth.
We should like to stress that the F1 score depends on thresholding method, which is quite simple in our current 
implementation and could be improved.
Finally, a number of modern methods including the top three in the CDnet ranking as of May 2018 
are ``supervised'', in the sense that they derive megabytes of a model from the test set and then apply the model to the test set, 
which constitutes ``double dipping''. 
With these caveats in mind, the performance seems rather respectable.


\begin{table*}[h!]
\begin{tabular}{lcccc}
\toprule
\texttt{Video sequence} & \texttt{MS-SSIM} & \texttt{F1-score} & \texttt{Recall} & \texttt{Precision} \\
\midrule
badWeather/blizzard                     & 0.990 {\tiny (0.011)} & 0.752 {\tiny (0.164)} & 0.901 {\tiny (0.102)} & 0.675 {\tiny (0.189)} \\
badWeather/skating                      & 0.980 {\tiny (0.020)} & 0.872 {\tiny (0.147)} & 0.890 {\tiny (0.106)} & 0.891 {\tiny (0.173)} \\
badWeather/snowFall                     & 0.976 {\tiny (0.027)} & 0.601 {\tiny (0.236)} & 0.832 {\tiny (0.161)} & 0.505 {\tiny (0.242)} \\
badWeather/wetSnow                      & 0.979 {\tiny (0.022)} & 0.446 {\tiny (0.222)} & 0.838 {\tiny (0.122)} & 0.356 {\tiny (0.222)} \\
baseline/PETS2006                       & 0.983 {\tiny (0.005)} & 0.769 {\tiny (0.116)} & 0.963 {\tiny (0.040)} & 0.655 {\tiny (0.146)} \\
baseline/highway                        & 0.946 {\tiny (0.030)} & 0.886 {\tiny (0.070)} & 0.848 {\tiny (0.098)} & 0.934 {\tiny (0.049)} \\
baseline/office                         & 0.959 {\tiny (0.044)} & 0.652 {\tiny (0.192)} & 0.647 {\tiny (0.234)} & 0.745 {\tiny (0.200)} \\
baseline/pedestrians                    & 0.988 {\tiny (0.002)} & 0.930 {\tiny (0.083)} & 0.989 {\tiny (0.017)} & 0.887 {\tiny (0.115)} \\
dynamicBackground/boats                 & 0.794 {\tiny (0.022)} & 0.316 {\tiny (0.171)} & 0.209 {\tiny (0.138)} & 0.890 {\tiny (0.121)} \\
dynamicBackground/canoe                 & 0.758 {\tiny (0.031)} & 0.692 {\tiny (0.196)} & 0.561 {\tiny (0.206)} & 0.994 {\tiny (0.049)} \\
dynamicBackground/fall                  & 0.824 {\tiny (0.042)} & 0.274 {\tiny (0.182)} & 0.291 {\tiny (0.114)} & 0.430 {\tiny (0.348)} \\
dynamicBackground/fountain01            & 0.919 {\tiny (0.016)} & 0.245 {\tiny (0.140)} & 0.336 {\tiny (0.162)} & 0.208 {\tiny (0.133)} \\
dynamicBackground/fountain02            & 0.957 {\tiny (0.003)} & 0.785 {\tiny (0.143)} & 0.702 {\tiny (0.172)} & 0.925 {\tiny (0.101)} \\
dynamicBackground/overpass              & 0.935 {\tiny (0.015)} & 0.644 {\tiny (0.154)} & 0.584 {\tiny (0.150)} & 0.778 {\tiny (0.223)} \\
intermittentObjectMotion/abandonedBox   & 0.997 {\tiny (0.004)} & 0.563 {\tiny (0.284)} & 0.505 {\tiny (0.295)} & 0.724 {\tiny (0.283)} \\
intermittentObjectMotion/parking        & 0.945 {\tiny (0.106)} & 0.230 {\tiny (0.297)} & 0.190 {\tiny (0.293)} & 0.868 {\tiny (0.243)} \\
intermittentObjectMotion/sofa           & 0.979 {\tiny (0.013)} & 0.518 {\tiny (0.213)} & 0.501 {\tiny (0.226)} & 0.585 {\tiny (0.266)} \\
intermittentObjectMotion/streetLight    & 0.999 {\tiny (0.002)} & 0.339 {\tiny (0.277)} & 0.294 {\tiny (0.280)} & 0.756 {\tiny (0.347)} \\
intermittentObjectMotion/tramstop       & 0.977 {\tiny (0.036)} & 0.393 {\tiny (0.268)} & 0.293 {\tiny (0.224)} & 0.727 {\tiny (0.341)} \\
intermittentObjectMotion/winterDriveway & 0.970 {\tiny (0.045)} & 0.394 {\tiny (0.197)} & 0.930 {\tiny (0.163)} & 0.286 {\tiny (0.177)} \\
lowFramerate/port\_0\_17fps             & 0.988 {\tiny (0.013)} & 0.223 {\tiny (0.179)} & 0.557 {\tiny (0.275)} & 0.187 {\tiny (0.183)} \\
lowFramerate/tramCrossroad\_1fps        & 0.995 {\tiny (0.007)} & 0.758 {\tiny (0.146)} & 0.934 {\tiny (0.061)} & 0.663 {\tiny (0.167)} \\
lowFramerate/tunnelExit\_0\_35fps       & 0.979 {\tiny (0.024)} & 0.628 {\tiny (0.277)} & 0.836 {\tiny (0.129)} & 0.564 {\tiny (0.286)} \\
lowFramerate/turnpike\_0\_5fps          & 0.967 {\tiny (0.034)} & 0.736 {\tiny (0.185)} & 0.639 {\tiny (0.204)} & 0.947 {\tiny (0.044)} \\
nightVideos/bridgeEntry                 & 0.980 {\tiny (0.021)} & 0.098 {\tiny (0.068)} & 0.977 {\tiny (0.042)} & 0.053 {\tiny (0.040)} \\
nightVideos/busyBoulvard                & 0.995 {\tiny (0.006)} & 0.304 {\tiny (0.177)} & 0.623 {\tiny (0.240)} & 0.259 {\tiny (0.237)} \\
nightVideos/fluidHighway                & 0.935 {\tiny (0.067)} & 0.103 {\tiny (0.098)} & 0.946 {\tiny (0.077)} & 0.059 {\tiny (0.061)} \\
nightVideos/streetCornerAtNight         & 0.985 {\tiny (0.022)} & 0.281 {\tiny (0.164)} & 0.838 {\tiny (0.149)} & 0.188 {\tiny (0.153)} \\
nightVideos/tramStation                 & 0.986 {\tiny (0.018)} & 0.750 {\tiny (0.119)} & 0.895 {\tiny (0.120)} & 0.668 {\tiny (0.144)} \\
nightVideos/winterStreet                & 0.955 {\tiny (0.050)} & 0.202 {\tiny (0.113)} & 0.946 {\tiny (0.080)} & 0.119 {\tiny (0.075)} \\
shadow/backdoor                         & 0.984 {\tiny (0.005)} & 0.889 {\tiny (0.113)} & 0.924 {\tiny (0.057)} & 0.874 {\tiny (0.137)} \\
shadow/bungalows                        & 0.949 {\tiny (0.064)} & 0.553 {\tiny (0.279)} & 0.620 {\tiny (0.351)} & 0.670 {\tiny (0.248)} \\
shadow/busStation                       & 0.960 {\tiny (0.033)} & 0.733 {\tiny (0.151)} & 0.765 {\tiny (0.192)} & 0.775 {\tiny (0.173)} \\
shadow/copyMachine                      & 0.942 {\tiny (0.030)} & 0.571 {\tiny (0.131)} & 0.758 {\tiny (0.200)} & 0.528 {\tiny (0.236)} \\
shadow/cubicle                          & 0.983 {\tiny (0.012)} & 0.696 {\tiny (0.210)} & 0.757 {\tiny (0.193)} & 0.705 {\tiny (0.255)} \\
shadow/peopleInShade                    & 0.968 {\tiny (0.031)} & 0.823 {\tiny (0.226)} & 0.992 {\tiny (0.015)} & 0.754 {\tiny (0.267)} \\
\bottomrule
\end{tabular}
\vskip 4mm
\caption{Results on changedetection.net.}
\label{tab:scores}
\end{table*}

\clearpage
In Table \ref{tab:DAEvsOMoGMF2}, we present a comparison similar to Table \ref{tab:DAEvsOMoGMF}, except that results of 
Algorithm~\ref{alg:SCDM} are obtained by using the smooth Geman-McLure loss function instead of subsampling with the non-smooth L1 norm.

\begin{table*}[h!]
\caption{Further results on  {\footnotesize\url{http://changedetection.net}}.}
\label{tab:DAEvsOMoGMF2}
\begin{center}
{\footnotesize
\begin{tabular}{lcccccc}
\toprule
Method / Category & Recall & Specificity & FPR
& FNR & Precision & F1     \\
\midrule
\textbf{Algorithm~\ref{alg:SCDM} (w/ Geman-McLure)}: & {} & {} & {} & {} & {} & {} \\
\midrule
badWeather        & 0.86733 & 0.98695 & 0.01305 & \textbf{0.13267} & 0.52229 & 0.62347 \\
baseline          & 0.85684 & \textbf{0.99078} & \textbf{0.00922} & 0.14316 & \textbf{0.77210} & \textbf{0.80254} \\
cameraJitter      & 0.59669 & \textbf{0.95909} & \textbf{0.04091} & 0.40331 & \textbf{0.56832} & \textbf{0.51461} \\
dynamicBackground & 0.46994 & \textbf{0.99627} & \textbf{0.00373} & 0.53006 & \textbf{0.63826} & \textbf{0.49071} \\
nightVideo        & \textbf{0.83068} & 0.87444 & 0.12556 & \textbf{0.16932} & 0.20758 & 0.29115 \\
shadow            & \textbf{0.76715} & \textbf{0.97409} & \textbf{0.02591} & \textbf{0.23285} & \textbf{0.61857} & \textbf{0.66833} \\
\midrule
Overall           & 0.73144 & \textbf{0.96360} & \textbf{0.03640} & 0.26856 & \textbf{0.55452} & \textbf{0.56514} \\
\midrule
\textbf{OMoGMF} \cite{Yong18}: & {} & {} & {} & {} & {} & {} \\
\midrule
badWeather        & \textbf{0.86871} & 0.98939 & 0.01061 & 0.13129 & \textbf{0.57917} & \textbf{0.67214} \\
baseline          & \textbf{0.89943} & 0.98289 & 0.01711 & \textbf{0.10057} & 0.62033 & 0.72611 \\
cameraJitter      & \textbf{0.85954} & 0.90739 & 0.09261 & \textbf{0.14046} & 0.30567 & 0.44235 \\
dynamicBackground & \textbf{0.87655} & 0.86383 & 0.13617 & \textbf{0.12345} & 0.08601 & 0.15012 \\
nightVideo        & 0.75607 & 0.92372 & 0.07628 & 0.24393 & 0.23252 & 0.31336 \\
shadow            & 0.55772 & 0.80276 & 0.03057 & 0.27562 & 0.40539 & \textbf{0.37450} \\
\midrule
Overall           & \textbf{0.80300} & 0.91166 & 0.06056 & \textbf{0.16922} & 0.37151 & 0.44643 \\
\midrule
\textbf{ST\_GRASTA} \cite{He11}: & {} & {} & {} & {} & {} & {} \\
\midrule
badWeather        & 0.26555 & \textbf{0.98971} & \textbf{0.01029} & 0.73445 & 0.45526 & 0.30498 \\
baseline          & 0.45340 & 0.98205 & 0.01795 & 0.54660 & 0.44009 & 0.42367 \\
cameraJitter      & 0.51138 & 0.91313 & 0.08687 & 0.48862 & 0.23995 & 0.31572 \\
dynamicBackground & 0.41411 & 0.94755 & 0.05245 & 0.58589 & 0.08732 & 0.13736 \\
nightVideo        & 0.42488 & \textbf{0.97224} & \textbf{0.02776} & 0.57512 & \textbf{0.24957} & 0.28154 \\
shadow            & 0.44317 & 0.96681 & 0.03319 & 0.55683 & 0.42604 & 0.41515 \\
\midrule
Overall           & 0.41875 & 0.96192 & 0.03808 & 0.58125 & 0.31637 & 0.31307 \\
\midrule
\textbf{RPCA\_FPCP} \cite{Rodriguez13}: & {} & {} & {} & {} & {} & {} \\
\midrule
badWeather        & 0.82546 & 0.84424 & 0.15576 & 0.17454 & 0.09950 & 0.16687 \\
baseline          & 0.73848 & 0.94733 & 0.05267 & 0.26152 & 0.29994 & 0.37900 \\
cameraJitter      & 0.74452 & 0.84143 & 0.15857 & 0.25548 & 0.18436 & 0.29024 \\
dynamicBackground & 0.69491 & 0.80688 & 0.19312 & 0.30509 & 0.03928 & 0.07134 \\
nightVideos       & 0.79284 & 0.85751 & 0.14249 & 0.20716 & 0.11797 & 0.19497 \\
shadow            & 0.72132 & 0.90454 & 0.09546 & 0.27868 & 0.26474 & 0.36814 \\
\midrule
Overall :         & 0.75292 & 0.86699 & 0.13301 & 0.24708 & 0.16763 & 0.24509 \\
\bottomrule
\end{tabular}}
\end{center}
\end{table*}

\clearpage
In Table \ref{tab:all-baseline-time}, we present a comparison of the mean run-time per frame of the methods discussed in the paper. Notice that this corresponds to a 
speed-up of up to the factor of 103.

\begin{table}[ht]
\caption{Mean processing time per input frame (in seconds) on the ``baseline/highway'' video-sequence from {\footnotesize\url{http://changedetection.net}}. Note, our implementation does not use any parallelisation at the moment. This was done on purpose to run on a machine serving multiple cameras simultaneously.}
\label{tab:all-baseline-time}
\begin{center}
{\footnotesize
\begin{tabular}{lc}
\toprule
Method & Mean time per frame     \\
\midrule
LRR\_FastLADMAP \cite{Lin11}               & 4.611 \\
MC\_GROUSE \cite{Balzano13}                & 10.621 \\
OMoGMF \cite{Meng13,Yong18}                & 0.123 \\
RPCA\_FPCP \cite{Rodriguez13}              & 0.504 \\
ST\_GRASTA \cite{He11}                     & 3.266  \\
TTD\_3WD \cite{Oreifej13}                  & 10.343 \\
Algorithm~\ref{alg:SCDM} (w/ Geman-McLure) & \textbf{0.103} \\
Algorithm~\ref{alg:SCDM} (w/ $L_1$ norm)   & 0.194 \\
\bottomrule
\end{tabular}}
\end{center}
\end{table}

\clearpage
\section{Additional Illustrations}
\label{appD}

\begin{figure}[h!]
\includegraphics[width=\textwidth]{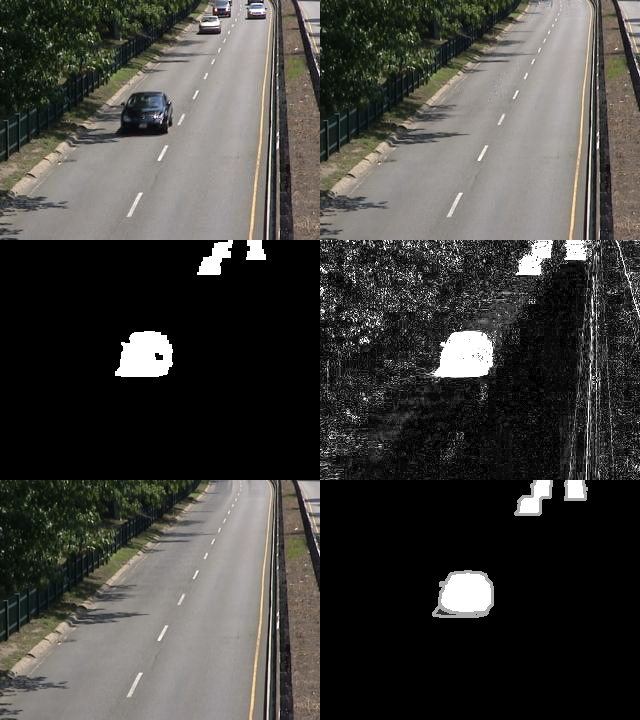}
\caption{One snapshot from the video baseline/highway (from the top left, clock-wise): one frame of the original video, our estimate of the background, our residuals prior to thresholding, the ground truth, an exponential smoothing of all frames prior to the current one with smoothing factor of 1/35, and finally, our Boolean map obtained by thresholding residuals.}
\end{figure}

\end{document}